\newif\ifarxiv\arxivfalse
\RequirePackage[dvipsnames]{xcolor}
\ifarxiv
	\documentclass[10pt, twoside, reqno]{amsart}
\else
	\RequirePackage{amsmath}
	\documentclass[smallextended]{svjour3}
	\usepackage[11pt]{extsizes}
	\usepackage[a4paper, total={16cm, 23cm},left=2cm,right=2cm,bottom=2.9cm,top=2cm]{geometry}
\fi
	\PassOptionsToPackage{noend}{algpseudocode}
	\PassOptionsToPackage{bookmarksdepth=4}{hyperref}
	
	\ifarxiv\else
		\makeatletter
			\let\cl@chapter\undefined
		\makeatother

		\smartqed
	\fi

	\usepackage{todonotes}
	\usepackage{bm}
	\usepackage{booktabs}
    \usepackage{pgffor,etoolbox}
	\usepackage{multirow}
\usepackage{subcaption}
\expandafter\def\csname ver@subfig.sty\endcsname{}
     \usepackage[normalem]{ulem}
    \usepackage{lscape}
    \usepackage{subfig}
    \usepackage{pgffor}
    \usepackage{etoolbox}
    \usepackage{graphicx}
    \usepackage{svg}
    
\usepackage[linesnumbered,ruled,vlined]{algorithm2e}
\RestyleAlgo{boxruled}
\SetInd{0.5em}{1em} % Explicit indentation setting
    \usepackage[%
		eqreset=section,
		thmreset=section,
	]{myPreamble}

	\makeatletter
		\newcommand\mynum[1]{$ˆ{\@fnsymbol{#1}}$}
	\makeatother
	
	\CreateNewTheorem[Fact]{fact}[dummythm]{Fact}
	\ifarxiv\else
	    
	\fi

	\makeatletter
		\newcommand\D[1][h]{%
			\@ifnextchar_{\operatorname D}{%
				\ifstrempty{#1}{\operatorname D}{\operatorname D_{#1}}%
			}
		}
	\makeatother
	
\Crefname{equation}{}{}

\Crefname{item}{}{}{}

\newcommand{\TheKeywords}{%
	Nonlinear system of equations, 
        Inexact Levenberg-Marquardt method,
        H\"{o}lder metric subregularity,
        Local superlinear convergence,
        Global convergence,
        Complexity analysis.%
}
\newcommand{\TheSubjclass}{%
        34A34, %{Nonlinear equations and systems, general}
	65K05, %{Math. Prog. Algorithm}
        90C30, %{Nonlinear Prog.}
        93E24. %{Least squares and related methods}
}
\newcommand{\TheTitle}{%
	Inexact Levenberg\textendash Marquardt methods under H\"{o}lder metric subregularity%
}
\newcommand{\TheShortTitle}{%
	Inexact Levenberg\textendash Marquardt methods under H\"{o}lder metric subregularity%
}
\newcommand{\TheShortAuthor}{%
	B. Symoens, M. Rahimi, and M. Ahookhosh%
}
\newcommand{\TheFunding}{%
	MR and MA acknowledge the support by the \emph{Research Foundation Flanders (FWO)} research project G081222N and
	\emph{UA BOF DocPRO4 projects} with ID 46929 and 48996.
}
\newcommand{\TheAddressUA}{%
		Department of Mathematics, University of Antwerp, Middelheimlaan 1, B-2020 Antwerp, Belgium}
\newcommand{\TheAbstract}{%

   This paper investigates two inexact Levenberg-Marquardt (LM) methods for solving systems of nonlinear equations. Both approaches compute approximate search directions by solving the LM linear system inexactly, subject to specific residual-based conditions. The first method uses an adaptive scheme to update the LM parameter, and we establish its local superlinear convergence under Hölder metric subregularity and local Hölder continuity of the gradient. The second method combines an inexact LM step with a nonmonotone quadratic regularization strategy. For this variant, we prove global convergence under the assumption of Lipschitz continuous gradients and derive a worst-case global complexity bound, showing that an approximate stationary point can be found in $\mathcal{O}(\epsilon^{-2})$ function and gradient evaluations. Finally, we justify the use of the LSQR algorithm for efficiently solving the linear systems involved, which is used in our numerical experiment on several nonlinear systems, including those appearing in real-world biochemical reaction networks, monotone and nonlinear equations, and image deblurring problems.
}

\begin{document}

\ifarxiv
	\title[\TheShortTitle]{\TheTitle}
	\author[\TheShortAuthor]{%
		Bas Symoens\textsuperscript{1},\ 
		Morteza Rahimi\textsuperscript{2}.\ 
		Masoud Ahookhosh\textsuperscript{2},\ and\ 

	}
	\thanks{\textsuperscript{1}\TheAddressKU}
	\thanks{\textsuperscript{2}\TheAddressMons}
	\thanks{\TheFunding}
	\keywords{\TheKeywords}
	\subjclass{\TheSubjclass}

% comment

		\begin{abstract}
			\TheAbstract
		\end{abstract}

		\maketitle
        
\else

	\journalname{}

	\title{\TheTitle\thanks{\TheFunding}}
	\titlerunning{\TheShortTitle}

	\author{%
            Bas Symoens\and
            Morteza Rahimi\and
		Masoud Ahookhosh%
	}
	\authorrunning{\TheShortAuthor}
    
    \institute{% 
		B. Symoens, M. Rahimi, and M. Ahookhosh
		\at
        \TheAddressUA.
        {\tt
            \{%
            \href{mailto:bas.symoens@uantwerpen.be}{bas.symoens},% 
            \href{mailto:morteza.rahimi@uantwerp.be}{morteza.rahimi},%
            \href{mailto:masoud.ahookhosh@uantwerp.be}{masoud.ahookhosh}%
            \}%
            \href{mailto:bas.symoens@uantwerpen.be, morteza.rahimi@uantwerp.be, masoud.ahookhosh@uantwerp.be}{@uantwerp.be}%
        }%
    }%

	\maketitle

	\begin{abstract}
		\TheAbstract
		\keywords{\TheKeywords}%
		\subclass{\TheSubjclass}%
	\end{abstract}
\fi

\vspace{-5mm}
%%%%%%%%%%%%%%%%%%%%%%%%%%%%%%%%%%%%%%%%%%%%%%%%%%%%%%%%%%%%%%%%%%%%%%%%%%%%

\section{Introduction}\label{sec:Introduction}

In this paper, we consider solving the system of nonlinear equations  
\begin{equation}
    F(x) = 0, \label{eq:nonequa}
\end{equation}
where \(F: \mathbb{R}^m \to \mathbb{R}^n\) is continuously differentiable and the solution set \(\Omega := \{x \in \mathbb{R}^m : F(x) = 0\}\) is nonempty. These problems appear in numerous scientific and engineering fields such as physics~\cite{eilenberger2012solitons,hasegawa2012plasma}, fluid mechanics~\cite{whitham2011linear}, chemical kinetics~\cite{aragon2015globally,aragon2018accelerating}, and applied mathematics (notably in solving differential equations)~\cite{ortega2000iterative}.

A classical strategy for tackling these problems is to minimize the nonlinear least-squares merit function
\begin{equation*}
\min_{x \in \mathbb{R}^m} \psi(x):= \frac{1}{2} \|F(x)\|^2,
\label{eq:meritfunc}
\end{equation*}
using iterative methods such as Newton, quasi-Newton, Gauss-Newton, adaptive regularization, and Levenberg-Marquardt algorithms~\cite{bellavia2010convergence,fischer2016globally,nocedal2006numerical,ortega2000iterative,yuan2015recent}.
Among these, the Levenberg-Marquardt (LM) method interpolates between gradient descent and Gauss–Newton by computing a search direction \(d_k\) through the subproblem
\begin{equation}
\min_{d \in \mathbb{R}^m} \phi_k(d) := \left\|\nabla F(x_k)^T d + F(x_k)\right\|^2 + \mu_k \|d\|^2,
\label{eq:def_phi_k}
\end{equation}
where \(\nabla F(x_k)\in \R^{m\times n}\) is the Jacobian of \(F\) at \(x_k\) and \(\mu_k > 0\) is a regularization parameter. This corresponds to finding a solution to the linear system 
\begin{equation}
\left(\nabla F(x_k) \nabla F(x_k)^T + \mu_k I\right) d_k = - \nabla F(x_k) F(x_k),
\label{eq:LM_direction}
\end{equation}
where $I\in \R^{m\times m}$ stands for the identity matrix. Then, the iteration is updated by \(x_{k+1}:=x_k+d_k\) and this process continues until a stopping criterion is met.
The choice of the regularization parameter \(\mu_k\) in~\eqref{eq:LM_direction} crucially affects the method’s performance; cf. \cite{izmailov2014newton,kelley1999iterative}. It is observed that with an appropriate choice of the parameter $\mu_k$, the LM method transitions between the gradient descent behavior when the iteration is far from the solution $x^*\in \Omega$, and that of the Gauss-Newton method as the iteration approaches $x^*$. Recently, in \cite{ahookhosh2019local}, the authors proposed the LMAR method with the adaptive regularization parameter
\begin{equation}
\mu_k := \xi_k \|F(x_k)\|^\eta + \omega_k \|\nabla F(x_k) F(x_k)\|^\eta,
\label{eq:muk}
\end{equation}
where \(\eta > 0\), \(\xi_k \in [\xi_{\min}, \xi_{\max}]\), and \(\omega_k \in [0, \omega_{\max}]\) for some positive constants \(\xi_{\min}, \xi_{\max}, \omega_{\max}\).

In recent decades, the use of error bounds to analyze the convergence rates of optimization methods has become a well-established technique. Under the assumption that \(\nabla F(x^*)\) is nonsingular, the LM method exhibits local quadratic convergence, and \(x^*\) is locally unique; see, e.g., \cite{bellavia2015strong,kanzow2005withdrawn,yamashita2001rate}. However, this is a restrictive assumption that often does not hold in practice.
This challenge can be addressed by employing \emph{local error bounds} \cite{ahookhosh2019local,behling2013effect,bellavia2015strong,fan2006convergence,fan2012modified,fan2005quadratic,yamashita2001rate}, which establish a quantitative relationship between the distance of iterations to the solution set and a computable residual.
In particular, a mapping \(F\) is said to be \emph{H\"{o}lder metric subregular} of order \(\delta\in (0,1]\) around \((x^*,0)\in \Omega\times\R^n\), if there exist some constants \(\beta > 0\) and \(r > 0\) such that
\begin{equation}
\beta \, \mathrm{dist}(x,\Omega) \le \|F(x)\|^\delta, \quad \forall x \in \mathbb{B}(x^*, r),
\label{eq:leb0}
\end{equation}
in which \(\mathrm{dist}(x,\Omega) = \inf_{y \in \Omega} \|x - y\|\) and \(\mathbb{B}(x^*,r)\) is a closed ball centered on \(x^*\) of radius \(r\). This condition implies that sufficiently small residuals ensure proximity to the solution set, thereby justifying the use of residual norms as a stopping criterion. Note that the nonsingularity of \(\nabla F(x^*)\) implies~\eqref{eq:leb0} with order $\delta=1$, the converse does not need to hold, allowing \(\Omega\) to be nonunique locally, a common situation in applications (see Section~\ref{sec.numapp}). This makes the local error bound a strictly weaker and more flexible assumption than nonsingularity.
In~\cite{yamashita2001rate}, the local quadratic convergence of the LM method was studied for \(\mu_k = \|F(x_k)\|^2\) and
under metric subregularity (\(\delta=1)\).  
Assuming H\"{o}lder metric subregularity, the local convergence of LMAR with different parameter $\mu_k$ has been studied in~\cite{ahookhosh2019local} and \cite{guo2015solving,zhu2016improved}.

In many applications, finding an initial point \(x_0\) sufficiently close to the solution \(x^*\) to guarantee fast convergence of the LM method is challenging, and in some cases, impossible. In the absence of such proximity, convergence of the LM method is not guaranteed, i.e., additional conditions are required to ensure global convergence. To address this issue, a globalization technique is typically incorporated in combination with the LM direction \(d_k\); see \cite{ahookhosh2020finding,ahookhosh2015globally,yamashita2001rate}.
An example of such a technique is \emph{quadratic regularization} where the adaptive parameter \(\mu_k\) is updated using a ratio 
\begin{equation}\label{eq:retiohat}
\overline{\rho}_{k}:=\frac{D_{k}-\psi\left(x_{k}+d_{k}\right)}{q_{k}(0)-q_{k}\left(d_{k}\right)},
\end{equation}
in which $D_{k}\geq \psi(x_k)$ is a nonmonotone term, and 
$q_{k}(d):=\frac{1}{2}\left\Vert\nabla F\left(x_{k}\right)^{T} d + F\left(x_{k}\right)\right\Vert^{2}.$
This ratio, as a measurement, plays a key role in selecting a direction that better aligns the model with the objective function, ensuring the global convergence of the LM  method to a stationary point of merit function $\psi$; see \cite{ahookhosh2012class,ahookhosh2020finding,gu2008incorporating}.

The key problem with the practical application of LM methods is the computational cost of solving the subproblem \eqref{eq:LM_direction}, especially when the dimension of the problem is large, and when the Jacobian matrix $\nabla F(x_k)$ is ill-conditioned. Moreover, investing substantial computational effort to compute a highly accurate search direction may be unreasonable in the early iterations, when the current iterate $x_k$ is still far from the solution. This motivates the quest for applying inexact techniques to solve this subproblem; see the early work \cite{dan2002convergence} followed by \cite{fan2004inexact,fan2011convergence,fischer2010inexactness,bao2019modified,wang2021convergence} under error bounds and singular value analysis. Recently, authors of~\cite{yin2024modified} investigated the convergence of a line search Inexact LM (ILM) method with $\mu_k=\xi_k\|F(x_k)\|^\eta$ under H\"{o}lder metric subregularity and local H\"{o}lder gradient continuity. These reasons are the primary motivation for our study to generate inexact variants of LM methods (cf. \cite{ahookhosh2019local,ahookhosh2020finding}) for large-scale problems under H\"{o}lder metric subregularity.

\subsection{{\bf Motivating large-scale nonlinear systems: biochemical reaction networks}}\label{sec:bio}

Biochemical reaction networks constitute a representative application in which systems of nonlinear equations arise naturally during the modeling process~\cite{fleming2012mass}. Consider a system with \(m\) molecular species involved in \(n\) reversible \emph{elementary reactions},
with \(m < n\). Let \(c \in \mathbb{R}_{++}^m\) denote the vector of molecular species concentrations. The deterministic dynamics of species concentrations are governed by
\begin{equation*}
\frac{dc}{dt} \equiv -f(c):= N \bigg( \exp\left(\ln(k_{f}) + F^{T} \ln(c)\right) - \exp\left(\ln(k_{r}) + R^{T} \ln(c)\right) \bigg),
\label{eq:dcdt2}
\end{equation*}
where \(N = R - F\) is the net \emph{stoichiometric matrix} (see, e.g.,\cite{ghaderi2021mathematical,ghaderi2020structural}), formed from the forward and reverse stoichiometric matrices \(F, R \in \mathbb{Z}_+^{m \times n}\) encoding the structure of the reactions, and \(k_f, k_r \in \mathbb{R}^n_{+}\) are the \emph{elementary kinetic parameters}.

A concentration vector \(c^* \in \mathbb{R}_{++}^m\) is called a \emph{steady state} if \(f(c^*):= -\frac{dc}{dt}(c^*) = 0\),
that is, the net rate of change in all species concentrations is zero. Identifying such steady states, whose existence is guaranteed under mild assumptions~\cite{gevorgyan2008detection,fleming2016conditions}, is fundamental to understanding the long-term behavior and functionality of biochemical networks.
Mathematically, this task reduces to solving the system of nonlinear equations
\begin{equation*}
F(x) := 
\begin{bmatrix}
\left[-\overline{N}\,|\, \overline{N}\right] \exp\left(k + [F\,|\,R]^{T} x\right) \\
L\, \exp(x) - l_0
\end{bmatrix}
= 0, \label{eq:steadyStateEquation}
\end{equation*}
where \(\overline{N}\) is a full-rank submatrix of \(N\) with ${\rm rank}(\overline{N})={\rm rank}(N)=\bar{r}$, \(k := [\ln(k_{f})^{T},\, \ln(k_{r})^{T}]^{T} \in \mathbb{R}^{2n}\), $x:=\ln(c)\in \R^m$, \(L \in \mathbb{R}^{(m - \bar{r}) \times m}\) is a basis for the left null space of \(N\), and \(l_0 \in \mathbb{R}_{++}^{m - \bar{r}}\) is a constant vector~\cite{haraldsdottir2016identification}.
In this setting, the function \(F\) is H\"{o}lder metrically subregular at \((x^*, 0)\), and the objective function \(\psi(x) := \frac{1}{2} \|F(x)\|^2\) is real analytic; see~\cite[Propositions~2.2.2 and 2.2.8]{krantz2002primer}.

\vspace{-3mm}
%%%%%%%%%%%%%%%%%%%%%%%%%%%%%%%%%%%%%%%%%%%%%%%%%%%%%%%%%%%%%%%%%%%%%%

\subsection{{\bf Contributions}}

Our contributions are threefold:

\vspace{-1mm}
\begin{description}
    \item[{\bf (i)}] {\bf Local convergence of ILLM method.}
        We propose an ILM method with adaptive regularization $\mu_k$, defined in \eqref{eq:muk} -- referred to as ILLM (\Cref{ALGORITHM-2:illm}) -- for computing approximate zeros of mappings that are H\"{o}lder metrically subregular. In this approach, the linear system \eqref{eq:LM_direction} is solved inexactly, yielding a direction $d_k$, where the level of inexactness is controlled by an upper bound on the residual. We show that ILLM attains a locally superlinear convergence rate under the assumptions of H\"{o}lder metric subregularity and local H\"{o}lder continuity of the gradient. \vspace{-2mm}
    \item[{\bf (ii)}] {\bf Global convergence of ILMQR method.}
        Building on the previous framework, we incorporate a nonmonotone quadratic regularization strategy to update the parameter \(\mu_k\), while maintaining control over the inexactness. This results in the ILM method with nonmonotone quadratic regularization -- referred to as ILMQR (\Cref{ALGORITHM-3:ILMQR}). We establish the global convergence of this method and provide a complexity analysis under Lipschitz continuity of the gradient and level boundedness of the merit function. \vspace{-2mm}
    \item[{\bf (iii)}] {\bf Numerical experiments.}
        We initially motivate the use of the LSQR algorithm, which enables efficient solution of the linear systems.
        Then, we validate the proposed methods in a variety of nonlinear systems, including models arising from biochemical reaction networks, monotone equations, and other nonlinear mappings. The methods exhibit superior performance on biochemical reaction network problems and demonstrate promising behavior across all other problems tested. Additionally, we apply our methods to an image deblurring problem, illustrating their capability to accurately reconstruct blurred and noisy images compared to some existing solvers.
\end{description}

\vspace{-6mm}
%%%%%%%%%%%%%%%%%%%%%%%%%%%%%%%%%%%%%%%%%%%%%%%%%%%%%%%%%%%%%%%%%%%%%%

\subsection{{\bf Organization}}

The remainder of the paper is organized as follows. Section~\ref{sec:prelim} introduces the notation and preliminaries. In Section~\ref{sec:Local}, we present the ILLM method and analyze its local convergence properties. Section~\ref{sec.ILMQR} introduces the ILMQR method, along with proofs of its global convergence and complexity. Section~\ref{sec.numapp} discusses iterative linear solvers and presents numerical experiments on biochemical reaction networks, (non)monotone mappings, and image deblurring. Finally, Section~\ref{sec:conclusion} summarizes our findings.

%%%%%%%%%%%%%%%%%%%%%%%%%%%%%%%%%%%%%%%%%%%%%%%%%%%%%%%%%%%%%%%%%%%%%%%%%%%%%%%%%%%%%%%%%
%%%%%%%%%%%%%%%%%%%%%%%%%%%%%%%%%%%%%%%%%%%%%%%%%%%%%%%%%%%%%%%%%%%%%%%%%%%%%%%%%%%%%%%%%

\section{{\bf Preliminaries}}\label{sec:prelim}

We begin this introductory section by reviewing some basic concepts and standard notation.
Given vectors $x, y \in \mathbb{R}^m$, we use $x^T$ to denote the transpose of $x$ and $x^T y$ for the inner product of $x$ and $y$. The Euclidean norm is defined as $\Vert x\Vert =\sqrt{x^Tx}$.
The closed ball centered on $x$ with radius $r > 0$ is denoted by $\mathbb{B}(x,r) := \{z \in \mathbb{R}^m: \Vert z-x\Vert \leq r\}$ and the distance from a point $x$ to a set $\Omega \subseteq \mathbb{R}^m$ is defined as $\mathrm{dist}(x,\Omega):=\inf\{\|x-z\| : z\in \Omega\}$.
Moreover, for each matrix $A = [a_{ij}] \in \mathbb{R}^{m\times n}$, we apply the induced matrix 2-norm, defined as $\Vert A\Vert :=\max_{\|x\|=1} \| Ax\|$.

We now recall the notions of H\"{o}lder metric subregularity and (local) H\"{o}lder gradient continuity. Let $F:\mathbb{R}^{m}\rightarrow\mathbb{R}^{n}$ be a continuously differentiable mapping.

\begin{definition}
\label{def:hmsr}
The mapping $F:\mathbb{R}^{m}\rightarrow\mathbb{R}^{n}$
is said to be H\"{o}lder metrically subregular of order $\delta\in (0,1]$
in $(\bar{x},\bar{y})$ with $\bar{y}=F(\bar{x})$ if there exist
some constants $r>0$ and $\beta>0$ such that
\[
\beta\,\mathrm{dist} \left(x,F^{-1}(\bar{y})\right)\leq\|\bar{y}-F(x)\|^{\delta},\quad\forall x\in\mathbb{B}(\bar{x},r).
\]
\end{definition}
For any solution $x^{*}\in\Omega$ of the system of nonlinear equations \eqref{eq:nonequa},
the H\"{o}lder metric subregularity of $F$ at $(x^{*},0)$ reduces to
\begin{equation}
\beta\,\mathrm{dist}(x,\Omega)\leq\|F(x)\|^{\delta},\quad\forall x\in\mathbb{B}(x^{*},r).\label{eq:errorBound}
\end{equation}
Therefore, this property provides an upper bound for the distance from
any point sufficiently close to the solution $x^{*}$ to the nearest
zero of the function.
%%%%%%%%%%%%%%
\begin{remark}
    There are several key points regarding H\"{o}lder metric subregularity worth highlighting.
    \begin{enumerate}[label=\textbf{\arabic*)}]
        \item[{\bf (a)}] 
            This property can also be defined analogously for set-valued mappings; see, e.g.,~\cite{kruger2015error}.
    
        \item[{\bf (b)}]  
            H\"{o}lder metric subregularity at $(x^{*},0)$ is also referred to as a
            \lq\lq H\"{o}lderian local error bound'' \cite{ngai2015global,vui2013global}.
            In the special case $\delta = 1$, it coincides with the standard notions of \lq\lq local error bound'' and \lq\lq metric subregularity''.
        \item[{\bf (c)}]  
            It is known that H\"{o}lder metric subregularity is closely related to the \L ojasiewicz inequalities\footnote{For definition of the \L ojasiewicz inequalities see \cite{lojasiewicz1993geometrie,kurdyka1998gradients}.}.
            In fact, the merit function $\psi(\cdot)=\frac{1}{2}\|F(\cdot)\|^{2}$ satisfies the \L ojasiewicz inequality if and only if the mapping $F$ is H\"{o}lder metrically subregular at $(x^{*},0)$; see \cite{ahookhosh2019local}.
            Moreover, Stanis\l{}aw \L{}ojasiewicz proved that real analytic functions, which commonly arise in applications, satisfy this property \cite{law1965ensembles}.
            Hence, a broad class of functions admits H\"{o}lder metric subregularity.
    
        \item[{\bf (d)}]  
            It is worth emphasizing that H\"{o}lder metric subregularity is weaker than the local error bound condition, when $\delta=1$; see \cite[Example 1]{ahookhosh2019local}.
            According to the Lyusternik\textendash Graves theorem (see, e.g.,~\cite[Theorem~5D.5]{dontchev2009implicit} or \cite[Theorem~1.57]{mordukhovich2006variational}), the nonsingularity of $\nabla F(x^{*})$ is equivalent to the \emph{metric regularity} of $F$ around $(x^*,0)$, which entails the existence of some positive constants $\beta$, $r$ and $s$ such that
            \[
            \beta\mathrm{dist}\left(x,F^{-1}(y)\right)\leq\|y-F(x)\|,\quad\forall x\in\mathbb{B}(x^{*},r),\forall y\in\mathbb{B}(0,s).
            \]
            Fixing $y=0$ yields the metric subregularity condition~\eqref{eq:errorBound} with $\delta=1$.
            Furthermore, by shrinking $r$ if necessary to ensure $\|F(x)\|<1$,
            this condition implies \eqref{eq:errorBound} for any $\delta\in(0,1]$.
    \end{enumerate}
\end{remark}
%%%%%%%%%%%%

Let $S\subseteq \R^m$ be a nonempty convex set and $\nu\in (0,1]$. The mapping $F:\mathbb{R}^{m}\rightarrow\mathbb{R}^{n}$
is said to be H\"{o}lder gradient continuous of order $\nu$ on \(S\) if there exists some modulus $L_H>0$ such that
$$\|\nabla F(y) - \nabla F(x)\| \leq L_H\|y-x\|^\nu,\quad\quad\forall x,y \in S,$$
which is equivalent to the H\"olderian condition
\begin{equation}\label{eq_Holderian_inequality}
    \left\Vert F(y)-F(x)-\nabla F(x)^T(y-x)\right\Vert \leq \frac{L_H}{1+\nu} \|y-x\|^{1+\nu},\quad \forall x,y\in S;
\end{equation}
see, \cite{yin2024modified,ahookhosh2019local}. For a given point $\overline{x}\in\R^m$, the mapping $F$ is said to be locally H\"{o}lder gradient continuous of order $\nu$ at $\overline{x}$ if there exists a radius $r>0$ such that $F$ is H\"{o}lder gradient continuous on \(\mathbb{B}(\overline{x},r)\). Moreover, the mapping $F$ is said to be Lipschitz gradient continuous if it is H\"{o}lder gradient continuous of order $\nu=1$ on $\R^m$.

%%%%%%%%%%%%%%%%%%%%%%%%%%%%%%%%%%%%%%%%%%%%%%%%%%%%%%%%%%%%%%%%%%%%%%%%%%%%%%%%%%%%%%%%%
%%%%%%%%%%%%%%%%%%%%%%%%%%%%%%%%%%%%%%%%%%%%%%%%%%%%%%%%%%%%%%%%%%%%%%%%%%%%%%%%%%%%%%%%%
\section{Local convergence of inexact Levenberg-Marquardt method}\label{sec:Local}
Let us begin by introducing ILLM and establishing its local convergence, which is the main contribution of this section.
By assuming certain properties, we aim to prove that this method converges superlinearly near a solution.

We present Algorithm ILLM for solving the nonlinear system~\eqref{eq:nonequa} through iterative refinement. The method combines the LM framework with an adaptive damping strategy and inexact linear solves to improve both robustness and computational efficiency.
In fact, instead of exactly solving the linear system~\eqref{eq:LM_direction}, ILLM computes an approximate solution $d_k$, referred to as the ILM direction, thereby reducing computational cost. 
This approximation is controlled by a residual condition to ensure sufficient accuracy.
Specifically, the residual
\begin{equation}\label{eq:LM_residual}
    \begin{array}{ll}
         r_{k}=r(x_k)\hspace{-2.5mm}
         &:= \left(\nabla F(x_{k})\nabla F(x_{k})^{T}+\mu_{k}I\right)d_{k}+\nabla F(x_{k})F(x_{k})         
         \hspace{0.9mm}= H_{k}d_{k} +\nabla \psi(x_{k}),
    \end{array}
\end{equation}
is required to satisfy the residual condition $\Vert r_{k}\Vert\leq\tau\mu_k\|d_k\|$ for a fixed $\tau\in(0,\tfrac{1}{2})$, where
$$H_{k}=H(x_{k}):=\nabla F(x_{k}) \nabla F(x_{k})^{T}+\mu_{k} I.$$
To summarize, our ILLM method in each iteration, first computes $\mu_{k}$ using \eqref{eq:muk}, and next, finds the direction $d_{k}$ as an approximate solution of the linear system \eqref{eq:LM_direction}, provided that the residual $r_{k}$ \eqref{eq:LM_residual} satisfies the condition $\Vert r_{k}\Vert\leq\tau\mu_k\|d_k\|$.
So, it updates the current step by moving in the direction $d_{k}$, which means $x_{k+1}=x_{k}+d_{k}$.
Eventually, after each update, it checks if a stopping criterion is met. The following is the process of ILLM:
%%%%%%%%%%%%%%%%%%%%%%
\RestyleAlgo{boxruled}
\begin{algorithm}[ht!]
    \caption{ILLM (Inexact Local Levenberg-Marquardt algorithm)}\label{ALGORITHM-2:illm}
    \DontPrintSemicolon
    \KwIn{$x_{0} \in \mathbb{R}^{m},\,\xi_{\max}>\xi_{\min}\geq 0,\, \omega_{\max}>\omega_{\min}\geq 0$ with $\xi_{\min}+\omega_{\min}>0,\, \eta>0,\, \tau\in(0,\tfrac{1}{2})$;}
    \Begin
    {
        $k:=0$;\;
        \While {$\|F(x_{k})\|>0$}
        {
            Choose $\xi_{k}\in[\xi_{\min}, \xi_{\max}],\, \omega_{k} \in[\omega_{\min},\omega_{\max }]$;\;
            Calculate $\mu_k$ by \eqref{eq:muk};\;
            Find the approximate solution $d_k$ by solving \eqref{eq:LM_direction} inexactly with $\|r_{k}\|\leq\tau\mu_k\|d_k\|$;\; 
            $x_{k+1}=x_{k} + d_{k}$;\;
            $k=k+1$;
        }
        \KwOut{$x_{k}$;}
    } 
\end{algorithm}
%%%%%%%%%%%%%%%

Now that we have presented the algorithm, let us focus on the conditions required for local convergence. These assumptions are used to prove the superlinear convergence of the method under certain constraints.
%%%%%%%%%%%%%%%%%%%%%%%%%%%%%%%%%%
\begin{ass}\label{as:local}
    Let us consider the system \eqref{eq:nonequa}. Given $x^*\in \Omega$, assume that
    \begin{enumerate}[label=\textbf{(A\arabic*)}]
        \item \label{as:local1}
            The mapping $F$ is H\"{o}lder metric subregular of order $\delta \in (0,1]$ at $(x^{*},0)$.
        
        \item \label{as:local2}
            The mapping $F$ is locally H\"{o}lder gradient continuous of order $\nu \in (0,1]$ at $x^{*}$.
    \end{enumerate}
\end{ass}
%%%%%%%%%%%%%%%%

\begin{remark}\label{rem:parameters}
    Let \cref{as:local} hold. Since $F$ is continuously differentiable, it is locally Lipschitz continuous. Therefore, there exist a sufficiently small radius $0<r<2$ and constants $\beta,L_{F},L_{H}>0$ such that: $(i)$ the mapping $F$ is Lipschitz on $\mathbb{B}(x^{*},r)$ with modulus $L_{F}$; $(ii)$ the mapping $F$ is  H\"{o}lder gradient continuous of order $\nu$ on $\mathbb{B}(x^{*},r)$ with modulus $L_{H}$; and $(iii)$ the mapping $F$ is H\"{o}lder metric subregular of order $\delta \in (0,1]$ at $(x^{*},0)$ with constants $r,\beta$, i.e., error bound \eqref{eq:errorBound} holds. Based on these parameters, we define the following constants:
    \begin{gather*}
        \tilde{r}\coloneqq 
        \begin{cases}
            \frac{r}{2} & \text{if } \xi_{\min}>0,\\
            \min\left\{\frac{r}{2}, 
            \left(\frac{\beta^{\frac{2}{\delta}}(1+\nu)^2}{2L_{H}^2}\right)^{\frac{1}{2\left(1+\nu-\frac{1}{\delta}\right)}}\right\}\hspace*{1cm} & \text{if $\xi_{\min}=0$},
        \end{cases}\hspace{1cm}
        \delta_1 \coloneqq 
        \begin{cases}
            \min\{1+\nu-\frac{\eta}{2\delta},1\} 
            & \mathrm{if\ }\xi_{\min}>0,\\[2mm]
            \min\left\{1+\nu-\frac{\eta}{2}\left(\frac{2}{\delta}-1\right),1\right\}\hspace*{.7cm} & \mathrm{if\ }\xi_{\min}=0,
        \end{cases}\\
            c_1 \coloneqq 
        \begin{cases}
            \frac{1}{1-\tau}\left(\frac{L_H}{2(1+\nu)\xi_{\min}^{\nicefrac{1}{2}}\beta^{\nicefrac{\eta}{(2\delta)}}} + 1\right)\hspace*{.7cm} & \mathrm{if\ }\xi_{\min}>0,\\[4mm]
            \frac{1}{1-\tau}
            \left(
            \frac{L_{H} 2^{\eta}}{2(1+\nu)\omega_{\min}^{\nicefrac{1}{2}}\beta^{\nicefrac{\eta}{\delta}}}
            + 1 \right) & \mathrm{if\ } \xi_{\min}=0,
        \end{cases}\hspace{1cm}
        \delta_2 \coloneqq 
        \min
        \left\{
        \delta\delta_1(1+\nu),
        \delta\left(\delta_1+\tfrac{\eta}{2}\right)
        \right\},
    \end{gather*}
    \begin{gather*}
        c_2:=\tfrac{1}{\beta}
        \left(
        \tfrac{c_1^{1+\nu}L_H}{1+\nu}
        + 
        \tfrac{\tau}{2}c_1(\xi_{\max} + \omega_{\max}L_F^{\eta})^{\tfrac{1}{2}}L_F^{\tfrac{\eta}{2}}
        +
        \left(\tfrac{L_H^2}{(1+\nu)^2} + (\xi_{\max}+\omega_{\max}L_F^\eta)L_F^\eta\right)^{\tfrac{1}{2}}
        \right)^{\tfrac{1}{\delta}}.
    \end{gather*}
\end{remark}
%%%%%%%%%%%%

The following preparatory lemma combines three well-known results that are essential for subsequent analysis. For details of the proof, see \cite[Lemma 2.2]{bao2019modified}.
%%%%%%%%%%%%%
\begin{lemma}\label{lem:matrix}
    Let $A$ be an $m\times n$ matrix and $\mu > 0$ be a constant. Then, the following inequalities hold:
    \begin{enumerate}[label=(\alph*)]
        \item $\|(A A^{T}+\mu I)^{-1}\|\leq \frac{1}{\mu}$;\label{eq:matrix1}
        \item $\|(A A^{T}+\mu I)^{-1}A\| \leq \frac{1}{2\sqrt{\mu}}$;\label{eq:matrix2}
        \item  $\|(A A^{T}+\mu I)^{-1}A A^{T}\|\leq 1$.\label{eq:matrix3}
    \end{enumerate}
\end{lemma}
%%%%%%%%%%%

%%%%%%%%%%%%%%%%%%%%%%%%%%%%%%%%%%%%%%%%%%%%
The following result plays a key role in the convergence analysis of ILLM throughout the remainder of this section. Specifically, it provides an upper bound on the norm of the ILM direction in terms of the distance from the current iterate to the solution set. 
%%%%%%%%%%%%%%%%%%%%%%%%%%%%%%%%%%%%%%%%%%
\begin{lemma}\label{lem:local_convergence}
    Let \cref{as:local} hold. Let $\eta\in \big(0,2\delta(1+\nu)\big)$ if $\xi_{\min}>0$,
    and $\delta>\tfrac{1}{1+\nu}$ and $\eta\in\left(0,\frac{2\delta(1+\nu)}{2-\delta}\right)$ if $\xi_{\min}=0$.
    If $x_k\in \mathbb{B}(x^*,\tilde{r})$ is an iteration generated by ILLM such that
    $x_k\not\in\Omega$, then
    \begin{equation}\label{eq:lem_dk}
        \|d_k\| \leq c_1\mathrm{dist}(x_k,\Omega)^{\delta_1},
    \end{equation}
    where the parameters $\tilde{r},\delta_1,c_1$ are the constants defined in \Cref{rem:parameters}.
\end{lemma}
\begin{proof}
    Invoking \Cref{lem:matrix}\ref{eq:matrix1} together with the residual condition $\Vert r_{k}\Vert\leq\tau\mu_k\|d_k\|$, it holds that
\begin{align*}
    \|d_k\| =& \left\| H_k^{-1}r_k - H_{k}^{-1} \nabla\psi(x_k)\right\| 
        \leq
        \left\|H_k^{-1}\right\|\left\|r_k\right\| + \left\|H_k^{-1}\nabla F(x_k)F(x_k)\right\| \\[1mm]
        \leq& 
        \tfrac{\|r_k\|}{\mu_k} + \left\|H_k^{-1}\nabla F(x_k)F(x_k)\right\|
        \leq
        \tau\|d_k\| + \left\|H_k^{-1}\nabla F(x_k)F(x_k)\right\|,
\end{align*}
implying
\begin{equation}
    \|d_k\| \leq \tfrac{1}{1-\tau}\left\|H_k^{-1}\nabla F(x_k)F(x_k)\right\|.
    \label{eq:dk_bound1}
\end{equation}
Let $\overline{x}_k \in \Omega$ denote the projection of $x_k$ onto $\Omega$, i.e., $\mathrm{dist}(x_k,\Omega)=\|\overline{x}_k-x_k\|$.
    Then,
    \begin{equation*}
        \|\overline{x}_k-x^*\| \leq \|\overline{x}_k - x_k\| + \|x_k -x^*\| \leq 2\|x_k -x^*\|\leq r,
    \end{equation*}
verifying $\overline{x}_k\in\mathbb{B}(x^*,r)$. 
Hence, from \Cref{lem:matrix}\ref{eq:matrix2} and \ref{eq:matrix3}, and H\"{o}lderian condition \eqref{eq_Holderian_inequality} on \(\mathbb{B}(x^*,r)\), we obtain
\begin{align*}
    \left\|H_k^{-1}\nabla F(x_k)F(x_k)\right\|
    \leq& \left\|H_k^{-1}\nabla F(x_k)\Big(F(x_k)+\nabla F(x_k)^{T}(\overline{x}_k - x_k)\Big)\right\|
    +
    \left\|H_k^{-1}\nabla F(x_k)\nabla F(x_k)^{T}(\overline{x}_k - x_k)\right\|\\[1mm]
    \leq& 
    \left\|H_k^{-1}\nabla F(x_k)\right\|\left\|F(x_k)+\nabla F(x_k)^{T}(\overline{x}_k - x_k)-F(\overline{x}_k)\right\|
    + 
    \left\|H_k^{-1}\nabla F(x_k)\nabla F(x_k)^{T}\right\|\left\|\overline{x}_k - x_k\right\|\\[1mm]
    \leq&  
    \tfrac{L_H}{2(1+\nu)\sqrt{\mu_k}}\|\overline{x}_k - x_k\|^{1+\nu} + \|\overline{x}_k - x_k\|
    = \tfrac{L_H}{2(1+\nu)\sqrt{\mu_k}}\mathrm{dist}(x_k,\Omega)^{1+\nu} + \mathrm{dist}(x_k,\Omega).
\end{align*}
Substituting this into \eqref{eq:dk_bound1} gives
\begin{equation}
    \|d_k\| \leq \tfrac{1}{1-\tau}\left(\tfrac{L_H}{2(1+\nu)\sqrt{\mu_k}}\mathrm{dist}(x_k,\Omega)^{1+\nu} + \mathrm{dist}(x_k,\Omega)\right).
    \label{eq:dk_bound2}
\end{equation}
For deriving inequality \eqref{eq:lem_dk}, two possible cases arise: (i) $\xi_{\min}>0$; (ii) $\xi_{\min}=0$.

In Case (i), by \Cref{as:local}~\ref{as:local1}, $\mu_k$ can be lower bounded as 
\begin{align*}
    \mu_k   = \xi_k\|F(x_k)\|^\eta + \omega_k \|\nabla F(x_k) F(x_k)\|^\eta
            \geq \xi_k\|F(x_k)\|^\eta 
            \geq \xi_{\min} \beta^{\eta/\delta}\mathrm{dist}(x_k,\Omega)^{\eta/\delta}.
\end{align*}
Substituting this lower bound into \eqref{eq:dk_bound2} yields
\begin{align*}
    \|d_k\|
    &\leq 
    \frac{1}{1-\tau}\left(\frac{L_H}{2(1+\nu)\xi_{\min}^{1/2}\beta^{\eta/(2\delta)}}\mathrm{dist}(x_k,\Omega)^{1+\nu-\frac{\eta}{2\delta}} + \mathrm{dist}(x_k,\Omega)\right)\\
    &\leq
    \frac{1}{1-\tau}\left(\frac{L_H}{2(1+\nu)\xi_{\min}^{1/2}\beta^{\eta/(2\delta)}} + 1\right)\mathrm{dist}(x_k,\Omega)^{\min\{1+\nu-\frac{\eta}{2\delta},1\}}, 
\end{align*}
thereby ensuring inequality \eqref{eq:lem_dk} in this case.

For Case (ii), $\xi_{\min}=0$ leads to $\omega_{\min}>0$. Thus,
\begin{align}\label{mu_lower_bound1}
    \mu_k 
    =\xi_k\|F(x_k)\|^\eta + \omega_k \|\nabla F(x_k) F(x_k)\|^\eta
    \geq \omega_{\min}\|\nabla F(x_k)F(x_k)\|^\eta.
\end{align}
Applying the properties in \Cref{as:local}, we conclude
\begin{align*}
    \frac{L^{2}_{H}}{(1+\nu)^2}\mathrm{dist}(x_k,\Omega)^{2(1+\nu)} 
    &\geq 
    \|F(x_k) + \nabla F(x_k)^T(\overline{x}_k-x_k)-F(\overline{x}_{k})\|^2\\
    &=
    \|F(x_k)\|^2 + 2(\overline{x}_k-x_k)^T\nabla F(x_k)F(x_k) + \|\nabla F(x_k)^T(\overline{x}_k-x_k)\|^2\\
    % &\geq
    % \beta^{\frac{2}{\delta}}\mathrm{dist}(x_k,\Omega)^{\frac{2}{\delta}} 
    % + 
    % 2(\overline{x}_k-x_k)^T\nabla F(x_k)F(x_k)\\
    &\geq 
    \beta^{\frac{2}{\delta}}\mathrm{dist}(x_k,\Omega)^{\frac{2}{\delta}} 
    -
    2\|\overline{x}_k-x_k\|\|\nabla F(x_k) F(x_k)\|,
\end{align*}
leading to
\begin{align*}
    \|\nabla F(x_k) F(x_k)\| 
    &\geq
    \frac{\beta^{\frac{2}{\delta}}}{2}\mathrm{dist}(x_k,\Omega)^{\frac{2}{\delta}-1} - \frac{L_H^2}{2(1+\nu)^2}\mathrm{dist}(x_k,\Omega)^{1+2\nu}\\
    &\geq
    \frac{\beta^{\frac{2}{\delta}}}{2}\mathrm{dist}(x_k,\Omega)^{\frac{2}{\delta}-1} - 
    \frac{L_H^2}{2(1+\nu)^2}\|x_k-x^*\|^{1+2\nu-\left(\frac{2}{\delta}-1\right)}\mathrm{dist}(x_k,\Omega)^{\frac{2}{\delta}-1}\\
    &\geq
    \frac{\beta^{\frac{2}{\delta}}}{2}\mathrm{dist}(x_k,\Omega)^{\frac{2}{\delta}-1} 
    - 
    \frac{L_H^2}{2(1+\nu)^2}
    \tilde{r}^{2\left( 1+\nu-\frac{1}{\delta}\right)}\mathrm{dist}(x_k,\Omega)^{\frac{2}{\delta}-1}\\
    &\geq
    \frac{\beta^{\frac{2}{\delta}}}{2}\mathrm{dist}(x_k,\Omega)^{\frac{2}{\delta}-1} 
    -
    \frac{\beta^{\frac{2}{\delta}}}{4}\mathrm{dist}(x_k,\Omega)^{\frac{2}{\delta}-1}
    =
    \frac{\beta^{\frac{2}{\delta}}}{4}\mathrm{dist}(x_k,\Omega)^{\frac{2}{\delta}-1}.
\end{align*}
Plugging this into \eqref{mu_lower_bound1}, the parameter $\mu_k$ is lower bounded as
\begin{equation*}
    \mu_k \geq \omega_{\min} \frac{\beta^{\frac{2\eta}{\delta}}}{4^\eta}\mathrm{dist}(x_k,\Omega)^{\left(\frac{2}{\delta}-1\right)\eta}.
\end{equation*}
Substituting this lower bound into \eqref{eq:dk_bound2} justifies that
\begin{align*}
    \|d_k\|
    &\leq
    \frac{1}{1-\tau}
    \left(
        \frac{L_H4^{\frac{\eta}{2}}}{2(1+\nu)\omega_{\min}^{\nicefrac{1}{2}}\beta^{\nicefrac{\eta}{\delta}}}\mathrm{dist}(x_k,\Omega)^{1+\nu-\frac{\eta}{2}\left(\frac{2}{\delta}-1\right)}
        +
        \mathrm{dist}(x_k,\Omega)
    \right)\\
    &\leq
    \frac{1}{1-\tau}
    \left(
        \frac{L_H 2^{\eta}}{2(1+\nu)\omega_{\min}^{\nicefrac{1}{2}}\beta^{\nicefrac{\eta}{\delta}}}
        +1
    \right) ~
    \mathrm{dist}(x_k,\Omega)^{\min\left\{1+\nu-\frac{\eta}{2}\left(\frac{2}{\delta}-1\right),1\right\}},
\end{align*}
adjusting our desired result.
\end{proof}
%%%%%%%%%%%

\Cref{pro:dist_decrease} indicates that the distance to the solution set contracts at a rate governed by a power of the previous distance. This reveals a locally superlinear convergence behavior of the ILLM.
%%%%%%%%%%%%%%%%%%%%%%%%%%%%%%%%%%%%
\begin{proposition}\label{pro:dist_decrease}
    Let \Cref{as:local} hold.
    Let $\eta\in (0,2\delta(1+\nu))$ if $\xi_{\min}>0$, and $\delta>\frac{1}{1+\nu}$ and $\eta\in\left(0,\frac{2\delta(1+\nu)}{2-\delta}\right)$ if $\xi_{\min}=0$.
    If $x_{k}, x_{k+1} \in\mathbb{B}(x^*,\tilde{r})$ are iterations generated by ILLM such that
    $x_k\not\in \Omega$, then
    \begin{equation}
        \mathrm{dist}(x_{k+1},\Omega)\leq c_2\mathrm{dist}(x_k,\Omega)^{\delta_2},\label{eq:dist_decrease}
    \end{equation}
    where the parameters $\tilde{r},\delta_2,c_2$ are the constants defined in \Cref{rem:parameters}.
\end{proposition}
\begin{proof}
    Let $\overline{x}_k\in\Omega$ denote the projection of $x_k$ onto $\Omega$, i.e., $\mathrm{dist}(x_k,\Omega)=\|\overline{x}_k-x_k\|$.
    It follows that
    \begin{align}
    \begin{split}\label{eq:muk_upperbound}
        \mu_k   &=
                \xi_k\|F(x_k)-F(\overline{x}_k)\|^\eta + \omega_k\|\nabla F(x_k)F(x_k) - \nabla F(x_k)^{T}F(\overline{x}_k)\|^\eta\\[2mm]
                &\leq
                \xi_k L_F^\eta\|\overline{x}_k - x_k\|^\eta + \omega_k L_F^\eta\|\nabla F(x_k)\|^\eta\|\overline{x}_k-x_k\|^\eta\\[2mm]
                &\leq
                (\xi_{\max}+\omega_{\max}\|\nabla F(x_k)\|^\eta)L_F^\eta \mathrm{dist}(x_k,\Omega)^\eta\\[2mm]
                &\leq
                (\xi_{\max}+\omega_{\max}L_F^\eta)L_F^\eta \mathrm{dist}(x_k,\Omega)^\eta.
    \end{split}
    \end{align}
    From the residual \eqref{eq:LM_residual}, we decompose the direction $d_k$ as
    \begin{align*}
        d_k = -H_k^{-1}\nabla\psi(x_k) + H_k^{-1} r_k = d_k^{LM} + H_k^{-1} r_k,
    \end{align*}
    where $d_k^{LM}$ denotes the exact LM direction.
    Utilizing this decomposition together with \Cref{lem:matrix}\ref{eq:matrix2}, the residual condition $\Vert r_{k}\Vert\leq\tau\mu_k\|d_k\|$, \Cref{lem:local_convergence}, and \eqref{eq:muk_upperbound}, we derive
    \begin{align}
    \begin{split}\label{eq:bound_newtonstep}
        \|F(x_k) + \nabla F(x_k)^{T}d_k\| 
        &= 
        \|F(x_k) + \nabla F(x_k)^{T}d_k^{LM} + \nabla F(x_k)^{T}H_k^{-1} r_k\|\\
        &\leq 
        \|F(x_k) + \nabla F(x_k)^{T}d_k^{LM}\| + \|\nabla F(x_k)^{T}H_k^{-1}\|\|r_k\|\\
        &\leq 
        \|F(x_k) + \nabla F(x_k)^{T}d_k^{LM}\| +\tfrac{\|r_k\|}{2\sqrt{\mu_k}}
        \leq
        \|F(x_k) + \nabla F(x_k)^{T}d_k^{LM}\| +\tfrac{\tau\sqrt{\mu_k}\|d_k\|}{2}\\
        &\leq
        \|F(x_k) + \nabla F(x_k)^{T}d_k^{LM}\| +\tfrac{\tau}{2}c_1(\xi_{\max} + \omega_{\max}L_F^{\eta})^{\tfrac{1}{2}}L_F^{\tfrac{\eta}{2}}\mathrm{dist}(x_k,\Omega)^{\delta_1+\tfrac{\eta}{2}}.
    \end{split}
    \end{align}
    Now, inasmuch as $d_k^{LM}$ is a unique solution to problem \eqref{eq:def_phi_k}, it satisfies
    \begin{align*}
        \|F(x_k) + \nabla F(x_k)^{T}d_k^{LM}\|^2
        &\leq 
        \phi_k(d_k^{LM})
        \leq 
        \phi_k(\overline{x}_k-x_k)
        =
        \|F(x_k) + \nabla F(x_k)^{T}(\overline{x}_k - x_k)\|^2 + \mu_k\|\overline{x}_k - x_k\|^{2}\\
        &\leq \frac{L_H^2}{(1+\nu)^2} \|\overline{x}_k - x_k\|^{2(1+\nu)}
        + (\xi_{\max}+\omega_{\max}L_F^\eta)L_F^\eta \mathrm{dist}(x_k,\Omega)^{\eta+2},
    \end{align*}
    where we invoked \eqref{eq_Holderian_inequality} and \eqref{eq:muk_upperbound} in the final step. Consequently,
    \begin{equation*}
        \|F(x_k) + \nabla F(x_k)^{T}d_k^{LM}\| \leq \left(
        \frac{L_H^2}{(1+\nu)^2} + (\xi_{\max}+\omega_{\max}L_F^\eta)L_F^\eta
        \right)^{\frac{1}{2}}\mathrm{dist}(x_k,\Omega)^{\min\{1+\nu, 1+\frac{\eta}{2}\}}.
    \end{equation*}
    Combining the above inequality with \eqref{eq:bound_newtonstep} gives
    \begin{equation}
        \|\nabla F(x_k)^Td_k + F(x_k)\| \leq 
        \overline{c}_{2}\mathrm{dist}(x_k,\Omega)^{\overline{\delta}_{2}}, \label{eq:lemma2_prelim}
    \end{equation}
    where
    %{\small
    \begin{align*}
        \overline{c}_{2} \coloneqq 
        \tfrac{\tau}{2}c_1(\xi_{\max} + \omega_{\max}L_F^{\eta})^{\tfrac{1}{2}}L_F^{\tfrac{\eta}{2}}
        +
        \left(
        \tfrac{L_H^2}{(1+\nu)^2} + (\xi_{\max}+\omega_{\max}L_F^\eta)L_F^\eta
        \right)^{\tfrac{1}{2}},\hspace{1cm}
        \overline{\delta}_{2} \coloneqq
        \min\left\{
        1+\nu,
        \delta_1+\tfrac{\eta}{2}
        \right\}.
    \end{align*}
    Invoking the H\"{o}lder metric subregularity and local H\"{o}lder gradient continuity, \Cref{lem:local_convergence}, and \eqref{eq:lemma2_prelim}, we infer 
    \begin{align*}
        \left(\beta\mathrm{dist}(x_k+d_k,\Omega)\right)^{\tfrac{1}{\delta}}
        &\leq
        \|F(x_k+d_k)\|
        \leq
        \left\|F(x_k+d_k) - F(x_k) - \nabla F(x_k)^{T}d_k\right\| + \left\|F(x_k) + \nabla F(x_k)^{T}d_k\right\|\\
        &\leq
        \tfrac{L_H}{1+\nu}\|d_k\|^{1+\nu} + \left\|F(x_k) + \nabla F(x_k)^{T}d_k\right\|
        \leq
        \tfrac{L_H}{1+\nu}c_1^{1+\nu}\mathrm{dist}(x_k,\Omega)^{\delta_1(1+\nu)} + \overline{c}_2\mathrm{dist}(x_k,\Omega)^{\overline{\delta}_{2}}\\
        &\leq
        \left( \tfrac{c_1^{1+\nu}L_H}{1+\nu} + \overline{c}_{2}\right)\mathrm{dist}(x_k,\Omega)^{\min\{\delta_1(1+\nu), \overline{\delta}_{2}\}},
    \end{align*}
    which confirms \eqref{eq:dist_decrease} and thereby adjusts our desired results.
\end{proof}
%%%%%%%%%%%

We now establish the main result of this section. In doing so, we identify a region where the parameter $\eta$ must be chosen to guarantee superlinear convergence.
To simplify, let us define the following constants:
\begin{align}\label{eq:sigma-barr}
    \sigma \coloneqq 
        \sum_{i=1}^\infty\left(\tfrac{1}{2}\right)^{\delta_1\delta_2^i},\hspace{2cm}
    \overline{r} \coloneqq
        \min\left\{
            \tfrac{1}{2}c_2^{\tfrac{-1}{\delta_2-1}},
            \left(
                \tfrac{\tilde{r}}{1+c_1+2^{\delta_1}c_1\sigma}
            \right)^{\tfrac{1}{\delta_1}}
        \right\},
\end{align}
\begin{align*}
    \overline{\eta}\coloneqq 
        \begin{cases}
            \min\left\{\tfrac{2\delta(1+\nu)}{1-\delta} -\tfrac{2}{1-\delta}, 2\delta(1+\nu) - \tfrac{2}{(1+\nu)}\right\}\hspace*{1cm} &  \mathrm{if\ } \xi_{\min}>0,~\delta<1,\\[3mm]
            \min\left\{\tfrac{\delta(1+\nu)}{1-\delta} -\tfrac{1}{1-\delta},\tfrac{2\delta(1+\nu)}{2-\delta}-\tfrac{2}{(1+\nu)(2-\delta)}\right\} & \mathrm{if\ } \xi_{\min}=0,~ \delta<1,\\[3mm]
            2(1+\nu)-\tfrac{2}{(1+\nu)} & \mathrm{if\ } \xi_{\min}\ge 0,~\delta=1,
        \end{cases}
\end{align*}
where the parameters $\tilde{r}, \delta_1, \delta_2, c_1, c_2$ are the constants defined in \Cref{rem:parameters}.

The next result ensures that the constants $\sigma$ and $\overline{r}$ are well defined under some mild conditions.

%%%%%%%%%%%%%%%%%%%%%%%%%%%%%%%%%%%
\begin{lemma}\label{lem:sigma-barr}
    Let $\tfrac{1-\delta}{\delta^{2}}<\nu$ if $\xi_{\min}>0$, and $\tfrac{(1-\delta)(2-\delta)}{\delta^{2}}<\nu$ if $\xi_{\min}=0$.
    If $\eta\in \left(\tfrac{2(1-\delta)}{\delta}, \overline{\eta}\right)$, then $\sigma$ is well defined, and $\overline{r}>0$, defined in \eqref{eq:sigma-barr}.
\end{lemma}
\begin{proof}
    It follows from the assumption that $\delta > \tfrac{1}{1+\nu}$ and $\eta$ is from a nonempty interval, i.e., \(\tfrac{2(1-\delta)}{\delta}<\overline{\eta}\).
    We first verify that $0<\delta_1$ and $\delta_2>1$ defined in \Cref{rem:parameters}.
    By the definition of $\delta_1$ and $\delta_2$, we observe that:
    \begin{align*}
        \delta_2 = 
        \begin{cases}
             \min\left\{\delta\left(1+\frac{\eta}{2}\right), \delta(1+\nu)\left(1+\nu-\frac{\eta}{2\delta}\right),\delta\left(1+\nu-\frac{\eta}{2\delta}+\frac{\eta}{2}\right)\right\} \hspace*{1cm}& ~~\mathrm{if\ } \xi_{\min}>0,\\[3mm]
             \min\left\{\delta\left(1+\frac{\eta}{2}\right), \delta(1+\nu)\left(1+\nu-\frac{\eta(2-\delta)}{2\delta}\right),\delta\left(1+\nu-\frac{\eta}{\delta}+\eta\right)\right\} & ~~\mathrm{if\ } \xi_{\min}=0.
        \end{cases}
    \end{align*}
    We first consider the case $\delta=1$. It holds that 
    $$\delta_2 = \min\left\{1+\tfrac{\eta}{2}, (1+\nu)(1+\nu-\tfrac{\eta}{2}), 1+\nu\right\}>1, \hspace{2cm} \delta_1=\min\{1+\nu-\tfrac{\eta}{2},1\}>0,$$
    as $\eta < 2(1+\nu)-\tfrac{2}{(1+\nu)}$.
    Let us now consider $\delta<1$.
    Inasmuch as $\eta > \tfrac{2(1-\delta)}{\delta}$, we have
    $\delta\left(1+\tfrac{\eta}{2}\right)>1.$
    In addition, the condition $\eta<\overline{\eta}$ guarantees that
    \begin{itemize}
        \item If $\xi_{\min}>0$, then \(
    \delta(1+\nu)\left(1+\nu-\frac{\eta}{2\delta}\right) >1\) and
    \(\delta\left(1+\nu-\tfrac{\eta}{2\delta}+\tfrac{\eta}{2}\right) >1\);
        
        \item If $\xi_{\min}=0$, then \(\delta(1+\nu)\left(1+\nu-\tfrac{\eta(2-\delta)}{2\delta}\right) >1\) and
    \(\delta\left(1+\nu-\tfrac{\eta}{\delta}+\eta\right) >1\).
    \end{itemize}
    Thus, in all cases, $\delta_2>1$ and $\delta_1>0$.
    It implies that $\sigma$ is finite and $\overline{r} > 0$, inasmuch as $\delta_1 \delta_2^i > i$ for sufficiently large $i$ and $\sum_{i=1}^{\infty} \left(\tfrac{1}{2}\right)^i = 1$. This completes the proof.
\end{proof}
%%%%%%%%%%%

\Cref{thm:superlinear} establishes the local superlinear convergence of the iterations distance \(\{\mathrm{dist}(x_k,\Omega)\}_{k\ge 0}\).

%%%%%%%%%%%%%%%%%%%%%%%%%%%%%%%%%%%%%%
\begin{theorem}[Superlinear convergence of ILLM]\label{thm:superlinear}
    Let \Cref{as:local} hold. Let
    $\tfrac{1-\delta}{\delta^{2}}<\nu$ if $\xi_{\min}>0$, and $\tfrac{(1-\delta)(2-\delta)}{\delta^{2}}<\nu$ if $\xi_{\min}=0$ and let
    $\eta\in \left(\tfrac{2(1-\delta)}{\delta} , \overline{\eta}\right)$.
    If the sequence $\{x_k\}_{k\ge 0}$ is generated by ILLM initialized at $x_0\in\mathbb{B}(x^*,\overline{r})$, then the following assertions hold:
    \begin{enumerate}[label=(\alph*)]
        \item\label{thm:superlinear_1}
            The sequence $\{\mathrm{dist}(x_k,\Omega)\}_{k\ge 0}$ converges superlinearly to $0$ with order $\delta_2$, i.e., \eqref{eq:dist_decrease} holds for all $k\ge 0$;
        \item\label{thm:superlinear_2}
            The sequence $\{x_k\}_{k\ge 0}$ is convergent to some
            $\overline{x}\in \Omega\cap \mathbb{B}(x^{*},\tilde{r})$;
        \item\label{thm:superlinear_3}
            If $\eta<2\delta\nu$ when $\xi_{\min}>0$, and $\eta<\tfrac{2\delta\nu}{2-\delta}$ when $\xi_{\min}=0$, then 
            the sequence $\{x_k\}_{k\ge 0}$ converges superlinearly with order $\delta_2$, for all $k$ sufficiently large, i.e.,
            \[
            \|x_k-\overline{x}\|\le 2c_1 c_2 \|x_{k-1}-\overline{x}\|^{\delta_2}.
            \]
    \end{enumerate}
    The parameters $\tilde{r},\delta_2,c_1,c_2$, and $\overline{r}$ are the constants defined in \Cref{rem:parameters} and \eqref{eq:sigma-barr}.
\end{theorem}
\begin{proof}
    For Assertion~\ref{thm:superlinear_1}, without loss of generality, assume that $x_{k}\not\in \Omega$ for all $k\ge 0$.
    Invoking \Cref{lem:sigma-barr}, it holds that $0<\delta_1\leq 1$ and $\delta_2>1$.
    It is sufficient to show that $\{x_k\}_{k\ge 0}\subseteq\mathbb{B}(x^*,\tilde{r})$.
    Then, the superlinear convergence of $\{\mathrm{dist}(x_k,\Omega)\}_{k\ge 0}$ to $0$ follows directly from \Cref{pro:dist_decrease}.
    To this end, we employ induction.
    For the case $k=0$, we have $x_0\in\mathbb{B}(x^*,\tilde{r})$ by assumption.
    Assume that $x_i\in\mathbb{B}(x^*,\tilde{r})$ for $i\leq k$.
    We indicate that $x_{k+1}\in\mathbb{B}(x^*,\tilde{r})$.
    For each $1\leq i\leq k$, it follows from \Cref{pro:dist_decrease} that
    \begin{align}
    \begin{split}\label{eq:upper_bound_dist}
        \mathrm{dist}(x_i,\Omega) 
        \leq
        c_2\mathrm{dist}(x_{i-1},\Omega)^{\delta_2}
        \leq
        \cdots
        \leq
        c_2^{\sum_{j=0}^{i-1}\delta_2^j}\mathrm{dist}(x_0,\Omega)^{\delta_2^i}
        \leq
        c_2^{\tfrac{\delta_2^i-1}{\delta_2-1}}\|x_0-x^*\|^{\delta_2^i}
        \leq
        \left(
        \tfrac{1}{2\overline{r}}
        \right)^{\delta_2^i-1}\overline{r}^{\delta_2^i}
        =
        2\overline{r}\left(\tfrac{1}{2}\right)^{\delta_2^i}.
    \end{split}
    \end{align}
    Following this inequality together with \eqref{eq:lem_dk}, we deduce
    \begin{align*}
        \|x_{k+1}-x^*\| &= \left\|x_{0} + \sum_{i=0}^k d_i -x^*\right\|
        \leq
        \|x_0-x^*\| + \sum_{i=0}^k\|d_i\|
        \leq
        \overline{r}
        +
        c_1\sum_{i=0}^k\mathrm{dist}(x_i,\Omega)^{\delta_1}\\
        &\leq
        \overline{r}^{\delta_1} + c_1\overline{r}^{\delta_1}
        +
        c_1\sum_{i=1}^k(2\overline{r})^{\delta_1}\left(\tfrac{1}{2}\right)^{\delta_1\delta_2^i}
        \leq
        (1+c_1)\overline{r}^{\delta_1}
        +
        c_1(2\overline{r})^{\delta_1}\sum_{i=1}^\infty\left(\tfrac{1}{2}\right)^{\delta_1\delta_2^i}\\
        &=
        (1+c_1+2^{\delta_1}c_1\sigma)\overline{r}^{\delta_1}
        \leq
        \tilde{r},
    \end{align*}
    ensuring the induction step. Hence, $\{x_k\}_{k\ge 0}\subseteq \mathbb{B}(x^*,\tilde{r})$.

    To justify Assertion \ref{thm:superlinear_2}, using \Cref{lem:local_convergence} together with \eqref{eq:upper_bound_dist}, we obtain
    \begin{align*}\label{eq:x_k_convergence}
        \sum_{k=0}^{\infty} \|x_{k+1}-x_{k}\| = \sum_{k=0}^{\infty} \|d_{k}\|
        \leq \sum_{k=0}^{\infty} c_1\mathrm{dist}(x_k,\Omega)^{\delta_1}
        \leq c_1 \sum_{k=0}^{\infty} (2\overline{r})^{\delta_1}\left(\tfrac{1}{2}\right)^{\delta_2^i\delta_{1}}
        =  c_1 (2\overline{r})^{\delta_1} \sigma<\infty.
    \end{align*}
    Consequently, the sequence \(\{x_k\}_{k\ge 0}\) converges to some $\overline{x}\in \Omega\cap\mathbb{B}(x^*,\tilde{r})$.
    
    In Assertion \ref{thm:superlinear_3}, it satisfies that $\delta_1=1$.
    On the other hand, superlinear convergence of $\{\mathrm{dist}(x_k,\Omega)\}_{k\ge 0}$ to $0$ implies that $\mathrm{dist}(x_{k+1},\Omega)\leq \tfrac{1}{2}\mathrm{dist}(x_k,\Omega)$, for all $k$ sufficiently large.
    Thus, for all $k$ sufficiently large, it yields that
    \begin{align*}
        \|x_{k} - \overline{x}\| 
        &= 
        \lim_{s\to\infty} \|x_{k} -x_{s}\| \leq \lim_{s\to\infty} \sum_{i=k}^{s-1} \|x_{i+1}-x_{i}\|
        \leq 
        \sum_{i=k}^{\infty} c_1\mathrm{dist}(x_i,\Omega)
        \leq 
        c_1 \sum_{i=k}^{\infty} \tfrac{1}{2^{i-k}}\mathrm{dist}(x_k,\Omega)\\
        &= 
        2c_{1}\mathrm{dist}(x_k,\Omega) \leq 2c_{1}c_{2}\mathrm{dist}(x_{k-1},\Omega)^{\delta_2}
        \leq 
        2c_{1}c_{2} \|x_{k-1} - \overline{x}\|^{\delta_2},
    \end{align*}
    ensuring our desired result.
\end{proof}

%%%%%%%%%%%%%%%%%%%%%%%%%%%%%%%%%%%%%%%%%%%%%%%%%%%%%%%%%%%%%%%%%%%%%%%%%%%%%%%%%%%%%%
\section{Globally convergent inexact Levenberg-Marquardt method}\label{sec.ILMQR}
%%%%%%%%%%%%%%%%%%%%%%%%%%%%%%%%%%%%%%%%%%%%%%%%%%%%%%%%%%%%%%%%%%%%%%%%%%%%%%%%%%%%%%

This section presents a globally convergent ILM method based on a quadratic regularization framework, followed by a detailed analysis of its convergence properties and iteration complexity.
% The LMQR method solves the quadratic subproblem \eqref{eq:def_phi_k} to compute a direction $d_{k}$ satisfying \eqref{eq:LM_direction}.
A central aspect of LM-type methods is choosing a direction that better aligns the model with the objective function. To quantify this alignment, a ratio is computed, which classifies these methods into two main categories: monotone and nonmonotone.
For monotone methods, the ratio of the actual to predicted reduction is defined as
\begin{equation}
\rho_{k}:=\frac{\psi\left(x_{k}\right)-\psi\left(x_{k}+d_{k}\right)}{q_{k}(0)-q_{k}\left(d_{k}\right)},
\end{equation}
in which the quadratic model $q_{k}:\mathbb{R}^{m}\rightarrow\mathbb{R}$ is given by
\(q_{k}(d):=\frac{1}{2}\left\Vert\nabla F\left(x_{k}\right)^{T} d + F\left(x_{k}\right)\right\Vert^{2}.\)
This ratio $\rho_{k}$ is then used to update the direction $d_{k}$ and the regularization parameter $\mu_{k}$.
In both line search and quadratic regularization settings, global convergence to a first-order stationary point of the merit function $\psi$ is guaranteed, leading to a monotonic sequence of function values, that is, $\psi\left(x_{k+1}\right) \leq \psi\left(x_{k}\right)$.
More recently, leveraging the quadratic model $q_{k}$, a nonmonotone LMQR method was introduced in \cite{ahookhosh2020finding}, which uses the ratio \eqref{eq:retiohat}, where the nonmonotone term $D_{k}$ is defined as
\begin{equation}\label{eq:nmt}
    D_{k}:=
    \left\{\begin{array}{ll}
         \psi(x_{0}) & \quad \mathrm{if\ } k=0,  \\[2mm]
         (1-\theta_{k-1}) \psi(x_{k})+\theta_{k-1} D_{k-1} \hspace{2mm}& \quad \mathrm{if\ } k \geq 1,
    \end{array}\right.
\end{equation}
with $\theta_{k-1} \in\left[\theta_{\min }, \theta_{\max }\right]$ and $0 \leq \theta_{\min } \leq \theta_{\max }<1$; see \cite{ahookhosh2012class,gu2008incorporating}.
In this formulation, the numerator represents the nonmonotone reduction, while the denominator remains the predicted reduction from the model.
In \cite{ahookhosh2020finding}, a new LM parameter was also proposed as a modified version of the adaptive regularization $\mu_{k}$ from \eqref{eq:muk}, defined as
\begin{equation}\label{eq:muhatk}
\overline{\mu}_{k}:=\max \big\{\mu_{\min }, \lambda_{k}\mu_{k}\big\},
\end{equation}
where $\eta >0,\, \mu_{\min}>0,\, \xi_{k} \in[\xi_{\min}, \xi_{\max}],\, \omega_{k} \in[\omega_{\min},\omega_{\max }]$ with $\xi_{\min} + \omega_{\min}>0$, and $\lambda_{k}$ is updated by the rule
$$
\lambda_{k+1}:= 
\begin{cases} 
\nu_{1} \lambda_{k} &\quad \mathrm{if\ } \overline{\rho}_{k}<\varrho_{1},\\
\lambda_{k}         &\quad \mathrm{if\ } \varrho_{1} \leq \overline{\rho}_{k}<\varrho_{2},\\
\nu_{2} \lambda_{k} &\quad \mathrm{if\ } \overline{\rho}_{k} \geq \varrho_{2},
\end{cases}
$$
in which $0<\nu_{2}<1<\nu_{1}$ and $0<\varrho_{1}<\varrho_{2}<1$.
A straightforward comparison between \eqref{eq:muk} and \eqref{eq:muhatk} reveals that $\overline{\mu}_{k}$ is uniformly bounded below and that the additional control $\lambda_{k}$ can enhance the numerical performance.

Let us emphasize that solving the linear system \eqref{eq:LM_direction} exactly in our LM method is commonly impractical for large-scale problems, especially when the current iteration is far from the solution. This motivates the development of the Inexact LMQR (ILMQR) method,  which seeks an approximate solution to the linear system 
\begin{equation}
\left(\nabla F(x_k) \nabla F(x_k)^T + \overline{\mu}_k I\right) d_k = - \nabla F(x_k) F(x_k),
\label{eq:LMQR_direction}
\end{equation}
such that the residual
\begin{equation*}\label{eq:LM_hat_residual}
    \begin{array}{ll}
         \overline{r}_{k}
         =
        \overline{r}(x_k)\hspace{-2.5mm}
         &:= 
         \left(\nabla F(x_{k})\nabla F(x_{k})^{T}+\overline{\mu}_{k}I\right)d_{k}+\nabla F(x_{k})F(x_{k})
         = 
         \overline{H}_{k}d_{k} +\nabla \psi(x_{k}),
    \end{array}
\end{equation*}
satisfies the residual condition $\Vert\overline{r}_{k}\Vert\leq\tau\overline{\mu}_k\|d_k\|$, for some constant $\tau\in(0,\tfrac{1}{2})$. Here,
$$\overline{H}_{k}=\overline{H}(x_{k}):=\nabla F(x_{k}) \nabla F(x_{k})^{T}+\overline{\mu}_{k} I.$$

In the proposed ILMQR algorithm, we first calculate $\overline{\mu}_{k}$ using \eqref{eq:muhatk}, then approximately solve the linear system \eqref{eq:LMQR_direction} to find the direction $d_{k}$ that satisfies the inexactness condition, $\Vert\overline{r}_{k}\Vert\leq\tau\overline{\mu}_k\|d_k\|$. Next, we evaluate the ratio $\overline{\rho}_{k}$ defined in \eqref{eq:retiohat}.
When $\overline{\rho}_{k}\approx 1$, the model aligns well with the objective function in $x_{k}$.
If $\overline{\rho}_{k} \geq \varrho_{1}$, the direction $d_{k}$ is accepted, which means $x_{k+1}=x_{k}+d_{k}$. Otherwise, the parameter $\lambda_{k}$ is increased by setting $\lambda_{k}=\nu_{1} \lambda_{k}$. If $\overline{\rho}_{k} \geq \varrho_{2}$, $\lambda_{k+1}$ is reduced by $\lambda_{k+1}=\max\{1, \nu_{2} \lambda_{k}\}$.
Finally, the algorithm checks the termination criteria.
This procedure is summarized in ILMQR.
%%%%%%%%%%%%%%%%%%%%%%
\RestyleAlgo{boxruled}
\begin{algorithm}[ht!]
    \LinesNumbered
    \DontPrintSemicolon
    \KwIn{$x_{0} \in \mathbb{R}^{m},\,\xi_{\max}>\xi_{\min}\geq 0,\, \omega_{\max}>\omega_{\min}\geq 0$ with $\xi_{\min}+\omega_{\min}>0$,\, $\mu_{\min }>0,\, \eta>0,$
    $0\leq \theta_{\min}\leq \theta_{\max}<1,\, 0<\varrho_{1}<\varrho_{2}<1,\, 0<\nu_{2}<1<\nu_{1},\, 0<\tau<\min\left\{\tfrac{1}{2},\tfrac{1-\varrho_{1}}{1+\varrho_{1}}\right\}$;}
    \Begin
    { 
        $k=0;$
        $\bar{\lambda}_{0}=1;$\;
        \While{a stopping criterion is not satisfied}
        {\label{beginouterloop}
            $\lambda_{k} = \bar{\lambda}_{k}$;\;
            Choose $\xi_{k}\in[\xi_{\min}, \xi_{\max}],\, \omega_{k} \in[\omega_{\min},\omega_{\max }]$, and $\theta_{k-1}\in\left[\theta_{\min }, \theta_{\max }\right]$;\;
            Calculate $\overline{\mu}_{k}$ and $D_{k}$ by \eqref{eq:muhatk} and \eqref{eq:nmt}, respectively;\;
            Find the approximate solution $d_k$ by solving \eqref{eq:LMQR_direction} inexactly with $\|\overline{r}_{k}\|\leq\tau\overline{\mu}_k\|d_k\|$;\;
            Compute $\overline{\rho}_{k}$ by \eqref{eq:retiohat};\;
            $s=0$;\;
            \While{$\overline{\rho}_k<\varrho_{1}$}
            {\label{begininnerloop}
                $s=s+1;\;
                \lambda_{k}=\nu_{1}^{s} \bar{\lambda}_{k}$;\;
                Calculate $\overline{\mu}_k$ by \eqref{eq:muhatk};\;
                Find the approximate solution $d_k$ by solving \eqref{eq:LMQR_direction} inexactly with $\|\overline{r}_{k}\|\leq\tau\overline{\mu}_k\|d_k\|$;\;
                Compute $\overline{\rho}_{k}$ by \eqref{eq:retiohat};
           }\label{endinnerloop}
            $s_{k}=s;$\;
            $x_{k+1}=x_{k}+d_{k};$\;
            \eIf{$\overline{\rho}_{k} \geq \varrho_{2}$}
            {
                $\bar{\lambda}_{k+1}= \max\{1, \nu_{2} \nu_{1}^{s_{k}} \bar{\lambda}_{k}\}$;
            }
            {
                $\bar{\lambda}_{k+1}=\nu_{1}^{s_{k}} \bar{\lambda}_{k}$;
            }
            $k=k+1$;
        }\label{endouterloop}
        \KwOut{$x_{k}$;}
    } 
\caption{ILMQR (Inexact Levenberg-Marquardt algorithm with Quadratic Regularization)\label{ALGORITHM-3:ILMQR}}
\end{algorithm}
%%%%%%%%%%%%%%%

In ILMQR, the loop starting from Line \ref{begininnerloop} to Line \ref{endouterloop} is referred to as the inner loop, while the loop from Line \ref{beginouterloop} to Line \ref{endouterloop} is called the outer loop. Moreover, one may use $\Vert F(x_{k})\Vert \leq \epsilon$ or $\Vert \nabla \psi(x_{k})\Vert \leq \epsilon$ as a stopping criterion in ILMQR.
To establish the global convergence of the sequence $\left\{x_{k}\right\}_{k\ge 0}$ generated by ILMQR to a stationary point of $\psi$, we assume that the following assumptions hold:
%%%%%%%%%%%%%%%%%%%%%%%%%%%%
\begin{ass}\label{as:global}
    Let us consider the system \eqref{eq:nonequa}. Let $x_{0}\in \mathbb{R}^m$ be arbitrary. Assume that
    \begin{enumerate}[label=\textbf{(A\arabic*)}]\setcounter{enumi}{2}
    \item \label{as:global1} The mapping $F$ is Lipschitz gradient continuous with modulus $L>0$ on $\mathbb{R}^m$;
    \item \label{as:global2} The lower level set $\Lambda(x_{0}):=\left\{x \in \mathbb{R}^{m} \mid \psi(x) \leq \psi(x_{0})\right\}$ is bounded, i.e., there exists $\Lambda>0$ that
    $\|x-x_{0}\|\leq \Lambda$ for all $x\in \Lambda(x_{0}).$ Let us define the constant
    $L_{0}:=L \Lambda + \|\nabla F(x_{0})\|.$
\end{enumerate}
\end{ass}
%%%%%%%%%

\Cref{lem:basic_ILMQR} provides some upper and lower bounds for the numerator and denominator of $\rho_{k}-1$, which can be employed to show well-definedness and convergence results of ILMQR.
%%%%%%%%%%%%%%%%%%%%%%%%%%%%%%%%%%%%
\begin{lemma}\label{lem:basic_ILMQR}
    Let \Cref{as:global} hold.
    If $x_{k}\in \Lambda(x_{0})$ and the direction $d_{k}$ solves the system \eqref{eq:LMQR_direction} inexactly such that $\|\overline{r}_{k}\|\leq \tau\overline{\mu}_{k}\|d_{k}\|$, then the following assertions hold:
    \begin{enumerate}[label=(\alph*)]
        \item\label{lem:basic_ILMQR_p1}
            $q_{k}(0)-q_{k}(d_{k})\geq \tfrac{\overline{\mu}_{k}}{2}\|d_{k}\|^{2};$
        \item\label{lem:basic_ILMQR_p2}
            $q_{k}(0)-q_{k}(d_{k})\geq \tfrac{1}{2\left(L_{0}^{2}+\overline{\mu}_{k}\right)} \left\|\overline{r}_{k}-\nabla \psi\left(x_{k}\right)\right\|^{2};$
        \item\label{lem:basic_ILMQR_p3}
            $q_{k}(0)-q_{k}(d_{k})\leq \left(\tfrac{1}{2} L_{0}^{2}+(1+\tau) \overline{\mu}_{k}\right) \|d_{k}\|^{2};$
        \item\label{lem:basic_ILMQR_p4}
            $|q_{k}(d_{k})-\psi(x_{k}+d_{k})| \leq \bigg(\frac{3L^{2}L_{0}^{2}\|F(x_{0})\|^{2}}{8\mu_{\min}^{2}(1-\tau)^{2}} + \frac{LL_{0}^{2}\|F(x_{0})\|}{2\mu_{\min}(1-\tau)} + \frac{L\|F(x_{0})\|}{2}\bigg)\|d_{k}\|^{2}$.
    \end{enumerate}
\end{lemma}
%%%%%%%%%%%
\begin{proof}
    \ref{lem:basic_ILMQR_p1} We observe that
    \begin{equation}\label{eq:basiclem1}
    \begin{gathered}
        q_{k}(0)-q_{k}(d_{k})
        = 
        \tfrac{1}{2} \|F(x_{k})\|^{2} -\tfrac{1}{2} \|\nabla F(x_{k})^{T}d_{k} + F(x_{k})\|^{2}
        = - \tfrac{1}{2} \|\nabla F(x_{k})^{T}d_{k}\|^{2} - \nabla \psi(x_{k})^{T}d_{k},\\
   \overline{r}_{k}^{T} d_{k} 
    = 
    \Big(\left(\nabla F(x_{k})\nabla F(x_{k})^{T}+\overline{\mu}_{k}I\right)d_{k}+\nabla F(x_{k})F(x_{k})\Big)^{T} d_{k}
    = 
    \|\nabla F(x_{k})^{T}d_{k}\|^{2} + \overline{\mu}_{k}\|d_{k}\|^{2} +  \nabla \psi(x_{k})^{T}d_{k}.
    \end{gathered}
    \end{equation}
    Moreover,
    \begin{align}\label{eq:Lbound_dHd}
        \big(\overline{r}_{k}-\nabla \psi(x_{k}) \big)^{T}d_{k} = d_{k}^{T} \overline{H}_{k} d_{k}\geq \lambda_{\min}(\overline{H}_{k}) \|d_{k}\|^{2}
        = \frac{1}{\lambda_{\max}(\overline{H}_{k}^{-1})} \|d_{k}\|^{2}
        = \frac{1}{\|\overline{H}_{k}^{-1}\|} \|d_{k}\|^{2}
        \geq \overline{\mu}_{k} \|d_{k}\|^{2},
    \end{align}
    where the last inequality uses \Cref{lem:matrix}~\ref{eq:matrix1}.
    Combining the inequalities in \eqref{eq:basiclem1} and using \eqref{eq:Lbound_dHd}, we come to
    \begin{align}\label{eq:Lbound_q0_qd}
    %\begin{array}{ll}
        q_{k}(0)-q_{k}(d_{k})
        &
        = 
        -\overline{r}_{k}^{T} d_{k} + \tfrac{1}{2} \overline{\mu}_{k}\|d_{k}\|^{2} + \tfrac{1}{2}
        \big(\overline{r}_{k}-\nabla \psi(x_{k}) \big)^{T}d_{k}\\[2mm]
        &\geq 
        \left(-\|\overline{r}_{k}\| + \tfrac{1}{2} \overline{\mu}_{k}\|d_{k}\|\right)\|d_{k}\| + \tfrac{\overline{\mu}_{k}}{2} \|d_{k}\|^{2}
        \geq 
        \tfrac{\overline{\mu}_{k}}{2} \|d_{k}\|^{2},\nonumber
    \end{align}
    where the last inequality uses $\|\overline{r}_{k}\|\leq \tau\overline{\mu}_{k}\|d_{k}\|$ with $\tau\in (0,\tfrac{1}{2})$, confirming Assertion~\ref{lem:basic_ILMQR_p1}.
    
    \ref{lem:basic_ILMQR_p2} By virtue of \eqref{eq:Lbound_q0_qd} together with residual condition, we obtain
    \begin{align}\label{eq:Lbound2_q0_qd}
        q_{k}(0)-q_{k}(d_{k}) 
        &\geq 
        \tfrac{1}{2}
        \big(\overline{r}_{k}-\nabla \psi(x_{k}) \big)^{T}d_{k}
        = 
        \tfrac{1}{2}\big(\overline{r}_{k}-\nabla \psi(x_{k}) \big)^{T} \overline{H}_{k}^{-1} \big(\overline{r}_{k}-\nabla \psi(x_{k}) \big)\nonumber\\[2mm]
        &\geq 
        \tfrac{1}{2}\lambda_{\min}(\overline{H}_{k}^{-1}) \left\|\overline{r}_{k}-\nabla \psi(x_{k})\right\|^{2} = 
        \tfrac{1}{2}\|\overline{H}_{k}\|^{-1} \left\|\overline{r}_{k}-\nabla \psi(x_{k})\right\|^{2}.
    \end{align}
    Furthermore, the Lipschitz continuity of $\nabla F$ and $x_{k}\in \Lambda(x_{0})$ imply that
    \begin{gather}
        \|\nabla F(x_{k})\| \leq \|\nabla F(x_{k}) - \nabla F(x_{0})\| + \|\nabla F(x_{0})\|
        \leq  L \| x_{k}-x_{0}\| + \|\nabla F(x_{0})\|
        \leq L_{0},\label{eq:bound_gradient}\\
        \|\overline{H}_{k}\|=\|\nabla F(x_{k}) \nabla F(x_{k})^{T}+\overline{\mu}_{k} I\| \leq \|\nabla F(x_{k}) \|^{2} + \overline{\mu}_{k} \leq L_{0}^{2}+\overline{\mu}_{k}.\label{eq:bound_H}
    \end{gather}
    Substituting \eqref{eq:bound_H} into \eqref{eq:Lbound2_q0_qd} ensures Assertion~\ref{lem:basic_ILMQR_p2}.

    \ref{lem:basic_ILMQR_p3} Utilizing \eqref{eq:Lbound_q0_qd} and \eqref{eq:bound_H} together with residual condition $\|\overline{r}_{k}\|\leq \tau\overline{\mu}_{k}\|d_{k}\|$, it holds that
    \begin{equation*}
    \begin{array}{ll}
        q_{k}(0)-q_{k}(d_{k})
        &= 
        -\overline{r}_{k}^{T} d_{k} + \tfrac{1}{2} \overline{\mu}_{k}\|d_{k}\|^{2} + \tfrac{1}{2}
        \big(\overline{r}_{k}-\nabla \psi(x_{k}) \big)^{T}d_{k}
        = 
        -\overline{r}_{k}^{T} d_{k} + \tfrac{1}{2} \overline{\mu}_{k}\|d_{k}\|^{2} + \tfrac{1}{2}
        d_{k}^{T} \overline{H}_{k} d_{k}\\[3mm]
        &\leq 
        \|\overline{r}_{k}\| \|d_{k}\| + \tfrac{1}{2} \overline{\mu}_{k}\|d_{k}\|^{2} +
         \tfrac{1}{2} \|\overline{H}_{k}\| \|d_{k}\|^{2}
        \leq 
        \tau \overline{\mu}_{k}\|d_{k}\|^{2} + \tfrac{1}{2} \overline{\mu}_{k} \|d_{k}\|^{2}+ \tfrac{1}{2} (L_{0}^{2}+\overline{\mu}_{k}) \|d_{k}\|^{2}\\[3mm]
        &=
        \left(\tfrac{1}{2} L_{0}^{2}+(1+\tau) \overline{\mu}_{k}\right) \|d_{k}\|^{2},
    \end{array}
    \end{equation*}
    verifying Assertion~\ref{lem:basic_ILMQR_p3}.

    \ref{lem:basic_ILMQR_p4} From \eqref{eq:Lbound_dHd}, it follows that
    \begin{align*}
        \overline{\mu}_{k}\|d_{k}\|^{2} 
        &\leq  
        \left(\overline{r}_{k}-\nabla \psi(x_{k})\right)^{T} d_{k}
        \leq 
        \left\|\overline{r}_{k}-\nabla \psi(x_{k})\right\| \|d_{k}\|
        \leq 
        \|\overline{r}_{k}\| \|d_{k}\| + \|\nabla\psi(x_{k})\| \|d_{k}\|
        \leq 
        \tau\overline{\mu}_{k} \|d_{k}\|^{2} + L_{0} \|F(x_{0})\|\|d_{k}\|,
    \end{align*}
    leading to
    \begin{equation}\label{eq:Ubound_d}
        \|d_{k}\| \leq \tfrac{L_{0}\|F(x_{0})\|}{\overline{\mu}_{k}(1-\tau)}\leq \tfrac{L_{0}\|F(x_{0})\|}{\mu_{\min}(1-\tau)}.
    \end{equation}
    In addition, from Lipschitzness of $\nabla F$ and $x_{k}\in \Lambda(x_{0})$ we get
    \begin{align}\label{eq:Ubound_Fx+d}
        \|F(x_{k}+d_{k})\| &\leq \|F(x_{k}+d_{k}) - F(x_{k}) - \nabla F(x_{k})^{T}d_{k}\| + \|F(x_{k})\|+ \|\nabla F(x_{k})^{T}d_{k}\|\nonumber\\[2mm]
        &\leq \tfrac{L}{2} \|d_{k}\|^{2} + \|F(x_{0})\|+\|\nabla F(x_{k})\| \|d_{k}\|
        \leq \tfrac{LL_{0}^{2}\|F(x_{0})\|^{2}}{2\mu_{\min}^{2}(1-\tau)^{2}} + \|F(x_{0})\| + \tfrac{L_{0}^{2}\|F(x_{0})\|}{\mu_{\min}(1-\tau)},
    \end{align}
    utilizing \eqref{eq:Ubound_d} in the last inequality.
    Finally, applying \eqref{eq:Ubound_d} and \eqref{eq:Ubound_Fx+d} we obtain
    \begin{align*}
        |q_{k}(d_{k})-\psi(x_{k}+d_{k})|
        &= 
        \bigg|\tfrac{1}{2}\Big\|\nabla F(x_{k})^{T}d_{k} + F(x_{k})\Big\|^{2} - \tfrac{1}{2} \|F(x_{k}+d_{k})\|^{2} \bigg| \nonumber\\[2mm]
        &= 
        \bigg|\tfrac{1}{2} \Big\|\nabla F(x_{k})^{T}d_{k} + F(x_{k}) - F(x_{k}+d_{k})\Big\|^{2}
        + \Big(\nabla F(x_{k})^{T}d_{k} + F(x_{k}) - F(x_{k}+d_{k})\Big)^{T} F(x_{k}+d_{k}) \bigg| \nonumber\\[2mm]
        &\leq 
        \tfrac{1}{2} \Big\|\nabla F(x_{k})^{T}d_{k} + F(x_{k}) - F(x_{k}+d_{k})\Big\|^{2}
        + \Big\|\nabla F(x_{k})^{T}d_{k} + F(x_{k}) - F(x_{k}+d_{k})\Big\| \|F(x_{k}+d_{k})\| \nonumber\\[2mm]
        &\leq 
        \tfrac{L^{2}}{8}\|d_{k}\|^{4} + \tfrac{L}{2} \|d_{k}\|^{2} \Big( \tfrac{LL_{0}^{2}\|F(x_{0})\|^{2}}{2\mu_{\min}^{2}(1-\tau)^{2}} + \tfrac{L_{0}^{2}\|F(x_{0})\|}{\mu_{\min}(1-\tau)} + \|F(x_{0})\|\Big)\\
        &\leq 
        \bigg(\tfrac{3L^{2}L_{0}^{2}\|F(x_{0})\|^{2}}{8\mu_{\min}^{2}(1-\tau)^{2}} + \tfrac{LL_{0}^{2}\|F(x_{0})\|}{2\mu_{\min}(1-\tau)} + \tfrac{L\|F(x_{0})\|}{2}\bigg)\|d_{k}\|^{2}.
    \end{align*}
    This completes the proof.
\end{proof}
%%%%%%%%%%%

The following proposition provides an essential result on the well-definedness of ILMOR, showing that the inner loop of ILMQR finitely concludes and the parameters $\overline{\mu}_{k}$ and $s_{k}$ are upper-bounded. For the sake of simplicity, let us define the following constants:
$$
\beta_{0}:=\tfrac{\nu_{1}}{1-\tau-\varrho_{1}(1+\tau)}\left(\tfrac{1}{2}(1+\varrho_{1}) L_{0}^{2}+\tfrac{3L^{2}L_{0}^{2}\|F(x_{0})\|^{2}}{8\mu_{\min}^{2}(1-\tau)^{2}} + \tfrac{LL_{0}^{2}\|F(x_{0})\|}{2\mu_{\min}(1-\tau)} + \tfrac{L\|F(x_{0})\|}{2}\right),
\hspace{1cm}
\beta_{1}=\max\{\mu_{\min},\beta_{0}\},
$$
$$
\tilde{\beta}:=\tfrac{\big(\xi_{\max}+L_{0}^{\eta}\omega_{\max}\big) \|F(x_{0})\|^{\eta}}{(\xi_{\min}+\omega_{\min})\epsilon^{\eta}}\max\big\{(\xi_{\min}+\omega_{\min})\epsilon^{\eta}, \mu_{\min}, \beta_{0}\big\}.
$$

%%%%%%%%%%%%%%%%%%%%%%%%%%%%%%%%%%%%%%%%%%%%%%%%%%%%%%%%%%%%%%%%%%%%%%%%%%
\begin{proposition}[Well-definedness of ILMQR]\label{pro:well_defineILMQR}
Let \Cref{as:global} hold.
If the infinite sequence $\left\{x_{k}\right\}_{k\ge 0}$ is generated by ILMQR, then the following assertions hold:
\begin{enumerate}[label=(\alph*)]
    \item\label{pro:well_defineILMQR_p1}
         In each iteration, the inner loop is terminated in a finite number of steps (i.e. $s_{k}<\infty$), $x_{k}\in \Lambda(x_{0})$, and
        $$\psi(x_{k})\leq D_{k}\leq D_{k-1},~~\text{for}~~ k\geq 1.$$
    \item\label{pro:well_defineILMQR_p2}
        If $\|F(x_{k})\|>\epsilon$ and $\|\nabla F(x_{k})F(x_{k})\|>\epsilon$ for all $k=0,1,\ldots,\tilde{k}$, for some $\tilde{k}\geq 0$, then
        \begin{equation*}\label{eq:Ubound_mu_beta}
        \overline{\mu}_{k} \leq \tilde{\beta},~~~k=0,1,\ldots,\tilde{k}.
        \end{equation*}
        \begin{equation}\label{eq:Ubound_s}
            s_{k} \leq \tfrac{\log(\beta_{1})-\log\big((\xi_{\min}+\omega_{\min})\epsilon^{\eta}\big)}{\log(\nu_{1})},~~~k=0,1,\ldots,\tilde{k}.
        \end{equation}
\end{enumerate}
\end{proposition}
%%%%%%%%%%%%%%%%%
\begin{proof}
    We employ induction to verify Assertion \ref{pro:well_defineILMQR_p1}.
    For the case $k=0$,
    inasmuch as $\nu_{1}>1$, $\overline{\mu}_{0}=\nu_{1}^{s}\bar{\lambda}_{0} \mu_{0}=\nu_{1}^{s}\mu_{0}$,
    for sufficiently large $s$.
    Invoking \Cref{lem:basic_ILMQR} \ref{lem:basic_ILMQR_p1} and \ref{lem:basic_ILMQR_p4}, and using $D_{0}=\psi(x_{0})$, it holds that
    $$
    |\overline{\rho}_{0}-1|=|\rho_{0}-1|\leq \frac{\tilde{c}}{\nu_{1}^{s}\mu_{0}}\to 0,~~~\text{as}~~s\to\infty,
    $$
    where
    $\tilde{c}=\tfrac{3L^{2}L_{0}^{2}\|F(x_{0})\|^{2}}{4\mu_{\min}^{2}(1-\tau)^{2}} + \tfrac{LL_{0}^{2}\|F(x_{0})\|}{\mu_{\min}(1-\tau)} + L\|F(x_{0})\|$.
    Hence, there exists $\hat{s}\in \mathbb{N}$ such that $\overline{\rho}_{0}\geq \varrho_{1}$ for some approximate solution $d_0$ of system \eqref{eq:LMQR_direction} with parameter $\overline{\mu}_{0}= \nu_{1}^{\hat{s}}\mu_{0}$.
    Thus, if the algorithm enters the inner loop, it exits after at most $\hat{s}$ iterations, i.e., $s_{0}\leq \hat{s}$.
    
    Suppose that $s_{k-1}<\infty$ and $\psi(x_{k-1}) \leq D_{k-1}\leq  D_{k-2}\leq\ldots\leq D_{0}= \psi(x_{0})$ (induction hypothesis).
    We now show that in iteration $k$, meaning the inner loop terminates in finite steps, $x_{k}\in \Lambda(x_{0})$, and $\psi(x_{k}) \leq D_{k}\leq  D_{k-1}$.
    By the induction hypothesis, $\overline{\rho}_{k-1}\geq \varrho_{1}>0$, implying $$\psi(x_{k})=\psi(x_{k-1}+d_{k-1})\leq D_{k-1}\leq\psi(x_{0}).$$
    Consequently, $x_{k}\in \Lambda(x_{0})$ and $\psi(x_{k})\leq D_{k}\leq D_{k-1}$, and as a result $\overline{\rho}_{k}\geq\rho_{k}$.
    To justify $s_{k}<\infty$, for sufficiently large $s$, we have $\overline{\mu}_{k}=\nu_{1}^{s}\bar{\lambda}_{k}\mu_{k}$
    and
    $$
    |\rho_{k}-1|\leq \frac{\tilde{c}}{\nu_{1}^{s}\lambda_{k}\mu_{k}}\to 0,~~~\text{as}~~s\to\infty,
    $$
    i.e., $\overline{\rho}_{k}\geq\rho_{k}\geq \varrho_{1}$ for some $\bar{s}\in \mathbb{N}$ and approximate solution $d_k$ of system \eqref{eq:LMQR_direction} with parameter $\overline{\mu}_{k}= \nu_{1}^{\bar{s}}\bar{\lambda}_{k}\mu_{k}$.
    Therefore, the inner loop terminates after at most $\bar{s}$ steps, i.e., $s_{k}\leq \bar{s}$, completing the proof of Assertion~\ref{pro:well_defineILMQR_p1}.  

    Now, let us consider Assertion \ref{pro:well_defineILMQR_p2}.
    We first prove $\overline{\mu}_{k}\leq \tilde{\beta}$.
    At any iteration $0\leq k\leq \tilde{k}$,
    after the inner loop, we have
    $\overline{\mu}_{k}=\max \big\{\mu_{\min }, \nu_{1}^{s_{k}}\bar{\lambda}_{k}\mu_{k}\big\}$
    with $s_{k}\geq 0$.  We distinguish three cases:
    (i) $\overline{\mu}_{k}=\mu_{\min}$;
    (ii) $\overline{\mu}_{k}=\nu_{1}^{s_{k}}\bar{\lambda}_{k}\mu_{k}$ with $s_{k}\geq 1$;
    (iii) $\overline{\mu}_{k}=\bar{\lambda}_{k}\mu_{k}$.
    
    In Case~(i), $\overline{\mu}_{k}=\mu_{\min}$ clearly leads to $\overline{\mu}_{k}\le\tilde{\beta}$ by virtue of
    $\tfrac{\big(\xi_{\max}+L_{0}^{\eta}\omega_{\max}\big) \|F(x_{0})\|^{\eta}}{(\xi_{\min}+\omega_{\min})\epsilon^{\eta}}\geq 1.$

    For Case~(ii),
    $\overline{\mu}_{k}=\nu_{1}^{s_{k}}\bar{\lambda}_{k}\mu_{k}$ with $s_{k}\geq 1$, let us consider the details of the previous iteration of the inner loop or of before that (when $s_{k}=1$). In this setting,
    $s=s_{k}-1$ and the vector $\widehat{d}_{k}$ is the approximate solution of system \eqref{eq:LMQR_direction} corresponding to the parameter
    $\widehat{\mu}_{k}=
    \max \big\{\mu_{\min }, \nu_{1}^{s_{k}-1}\bar{\lambda}_{k}\mu_{k}\big\}$ with residual satisfying $\|\overline{r}_k\|\leq \tau\widehat{\mu}_{k}\|\widehat{d}_{k}\|$.
    Using \Cref{lem:basic_ILMQR} \ref{lem:basic_ILMQR_p4}, \eqref{eq:Lbound_dHd}, and \eqref{eq:bound_gradient}, we obtain
    \begin{align*}
        2\left(\psi(x_{k}+\widehat{d}_{k}) - \psi(x_{k})\right) 
        &= \|F(x_{k}+\widehat{d}_{k})\|^{2} - \|F(x_{k})\|^{2} \nonumber\\[2mm]
        &= \|F(x_{k}+\widehat{d}_{k})\|^{2} - \|\nabla F(x_{k})^{T}\widehat{d}_{k} + F(x_{k})\|^{2} 
        + \|\nabla F(x_{k})^{T}\widehat{d}_{k}\|^{2} + 2 \nabla\psi(x_{k})^{T}\widehat{d}_{k}\nonumber\\[2mm]
        &= 2(\psi(x_{k}+\widehat{d}_{k}) - q_{k}(\widehat{d}_{k})) + \|\nabla F(x_{k})^{T}\widehat{d}_{k}\|^{2} + 2 \nabla\psi(x_{k})^{T}\widehat{d}_{k}\nonumber\\[2mm]
        &\leq \tilde{c}\|\widehat{d}_{k}\|^{2}+ \|\nabla F(x_{k})\|^{2}\|\widehat{d}_{k}\|^{2}
        + 2\overline{r}_k^{T}\widehat{d}_{k} - 2\widehat{\mu}_{k}\|\widehat{d}_{k}\|^{2}\nonumber\\[2mm]
        &\leq \big(\tilde{c} +L_{0}^{2}\big)\|\widehat{d}_{k}\|^{2}-2(1-\tau)\widehat{\mu}_{k}\|\widehat{d}_{k}\|^{2},
    \end{align*}
    i.e.,
    \begin{equation}\label{eq_ubound_mubar}
        \widehat{\mu}_{k} \leq \frac{1}{(1-\tau)\|\widehat{d}_{k}\|^{2}}\left(\psi(x_{k}) - \psi(x_{k}+\widehat{d}_{k})\right) + \frac{\tilde{c} +L_{0}^{2}}{2(1-\tau)}.
    \end{equation}
    Furthermore, from $\psi(x_{k}) \leq D_{k}$, Assertion \ref{pro:well_defineILMQR_p1}, $\overline{\rho}_{k} < \varrho_{1}$, and \Cref{lem:basic_ILMQR}~\ref{lem:basic_ILMQR_p3}, it follows that
    $$
    \psi(x_{k}) - \psi(x_{k}+\widehat{d}_{k}) \leq D_{k}-\psi(x_{k}+\widehat{d}_{k}) <\varrho_{1}\left(q_{k}(0)-q_{k}(\widehat{d}_{k})\right) \leq \varrho_{1}\left(\tfrac{1}{2} L_{0}^{2}+(1+\tau)\widehat{\mu}_{k}\right)\|\widehat{d}_{k}\|^{2}.
    $$
    Substituting this into \eqref{eq_ubound_mubar} leads to
    \begin{align*}
        \widehat{\mu}_{k} \leq \frac{1}{1-\tau-\varrho_{1}(1+\tau)}\left(\frac{1+\varrho_{1}}{2}L^{2}_{0} + \frac{\tilde{c}}{2}\right),
    \end{align*}
    by $\tau<\tfrac{1-\varrho_{1}}{1+\varrho_{1}}$. Therefore, in the final iteration of the inner loop, we have $\overline{\mu}_{k} \leq \nu_{1}\widehat{\mu}= \beta_{0}\le \tilde{\beta}$.

    In Case~(iii), $\overline{\mu}_{k}=\bar{\lambda}_{k}\mu_{k}$,
    where $\bar{\lambda}_{k}=\nu_{1}^{s_{k-1}}\bar{\lambda}_{k-1}$ or $\bar{\lambda}_{k}=\max\{1,\nu_{2}\nu_{1}^{s_{k-1}}\bar{\lambda}_{k-1}\}$ with $s_{k-1}\geq 0$. 
    On account of $\nu_{2}<1$ and $\bar{\lambda}_{k}\geq 1$, two scenarios arise for $\bar{\lambda}_{k}$: either $\bar{\lambda}_{k}=1$; or $1<\bar{\lambda}_{k}\leq \nu_{1}^{s_{k-1}}\bar{\lambda}_{k-1}$.
    For the first scenario, $\bar{\lambda}_{k}=1$, from $x_{k}\in \Lambda(x_{0})$ and \eqref{eq:bound_gradient}, one deduces that
    $$\overline{\mu}_{k}=\mu_{k}=\xi_{k}\|F(x_k)\|^{\eta}+\omega_{k}\|\nabla F(x_{k}) F(x_{k})\|^{\eta}\leq (\xi_{\max}+L_{0}^{\eta}\omega_{\max})\|F(x_{0})\|^{\eta}\leq \tilde{\beta}.$$
    Alternatively, consider the latter scenario, $1<\bar{\lambda}_{k}\leq \nu_{1}^{s_{k-1}}\bar{\lambda}_{k-1}$  with $s_{k-1}\geq 0$.
    We claim that there exists $0\leq j\leq k-1$ such that $s_{j}\geq 1$.
    On the contrary, if for each $0\leq j\leq k-1$, $s_{j}=0$, then
    $$1<\bar{\lambda}_k\leq\bar{\lambda}_{k-1} \leq \ldots \leq\bar{\lambda}_{0} =1,$$
    which leads to a contradiction.
    Thus, the claim is valid.
    Let $j$ be the smallest index such that $s_{j}> 0$ and
    $s_{k}=s_{k-1}=\ldots=s_{j+1}=0$.
    Then,
    $$1<\bar{\lambda}_k\leq\bar{\lambda}_{k-1} \leq \ldots \leq\bar{\lambda}_{j+1}\leq \nu_{1}^{s_{j}}\bar{\lambda}_{j}.$$
    Following the first and second cases, it holds that
    $$\bar{\lambda}_{k} (\xi_{\min}+\omega_{\min})\epsilon^{\eta}\leq \nu_{1}^{s_{j}}\bar{\lambda}_{j}\Big( \xi_{j}\|F(x_{j})\|^{\eta}+\omega_{j}\|\nabla F(x_{j}) F(x_{j})\|^{\eta}\Big)=\nu_{1}^{s_{j}}\bar{\lambda}_{j}\mu_{j} \leq\overline{\mu}_{j}\leq \beta_{1},$$
    by $\|F(x_{j})\|>\epsilon$, $\|\nabla F(x_{j}) F(x_{j})\|>\epsilon$, $\xi_{\min}\leq \xi_{j}$, and $\omega_{\min}\leq\omega_{j}$, i.e.,
    $$
    \bar{\lambda}_{k}\leq \frac{\beta_{1}}{(\xi_{\min}+\omega_{\min})\epsilon^{\eta}}.
    $$
    Utilizing this bound and the upper bound of $\mu_{k}$, we justify that
    $$\overline{\mu}_{k}=\bar{\lambda}_{k}\mu_{k}\leq \frac{\beta_{1}}{(\xi_{\min}+\omega_{\min})\epsilon^{\eta}}(\xi_{\max}+L_{0}^{\eta}\omega_{\max})\|F(x_{0})\|^{\eta}\leq \tilde{\beta}.$$
    
    Regarding the upper bound of $s_{k}$, without loss of generality, let the algorithm enter the inner loop at iteration $k$, meaning that $s_{k}\geq 1$ and $\overline{\mu}_{k}\leq \max\{\mu_{\min},\beta_{0}\}=\beta_{1}$.
    Then,
    $$
    \nu_{1}^{s_{k}} (\xi_{\min}+\omega_{\min})\epsilon^\eta \leq \nu_{1}^{s_{k}}\lambda_{k} \mu_{k}\leq
    \overline{\mu}_{k} \leq \beta_{1}.
    $$
    Taking the logarithm of both sides of this inequality leads to the desired result.
\end{proof}
%%%%%%%%%%%

The following theorem establishes the global convergence of the sequence $\{x_{k}\}$ produced by ILMQR to a first-order stationary point $x^{*}$ of the function $\psi$, that is $\nabla \psi(x^{*}) = 0$.

%%%%%%%%%%%%%%%%%%%%%%%%%%%%%%%%%%%%%%%%%%%%%%%%%%%%%%%%%%%%%%% 
\begin{theorem}[Global convergence]\label{thm:global_conv_ILMQR}
    Let \Cref{as:global} hold.
    If the sequence $\{x_{k}\}_{k\ge 0}$ is generated by ILMQR, then the sequences $\left\{\psi\left(x_{k}\right)\right\}_{k\ge 0}$ and $\left\{D_{k}\right\}_{k\ge 0}$ both converge to the same limit.
    Moreover, the algorithm either terminates after a finite number of iterations, satisfying $\|F(x_{k})\| \leq \epsilon$ or $\|\nabla \psi(x_{k})\| \leq \epsilon$, or any cluster point of that is a stationary point of the merit function $\psi$, i.e.
    \begin{equation*}\label{eq:critical}
        \lim _{k \rightarrow \infty}\|\nabla \psi(x_{k})\|=0.
    \end{equation*}
\end{theorem}
\begin{proof}
    For the first part, according to \Cref{pro:well_defineILMQR}~\ref{pro:well_defineILMQR_p1}, the sequence $\{D_{k}\}_{k\ge 0}$ is decreasing and lower bounded by $0$, i.e., it is convergent.
    Hence, it holds that
    $$D_{k}-D_{k-1} = (1-\theta_{k-1})(\psi(x_{k})-D_{k-1})\leq (1-\theta_{\max})(\psi(x_{k})-D_{k-1})\leq 0,$$
    since $0\leq \theta_{\min}\leq \theta_{k-1}\leq \theta_{\max}<1$.
    Consequently,
    $$\lim_{k\to\infty}(\psi(x_{k})-D_{k-1})=\lim_{k\to\infty}(D_{k}-D_{k-1})=0,$$
    ensuring $\displaystyle\lim_{k\to\infty}\psi(x_{k})=\lim_{k\to\infty}D_{k}.$

    For deriving the second part, assume that $\|F(x_{k})\| > \epsilon$ and $\|\nabla \psi(x_{k})\| >\epsilon$ for all $k\geq 0$.
    Obviously, after the inner loop for each iteration $k\geq 0$, we have $\overline{\rho}_{k}\ge \varrho_{1}$.
    Then, by \Cref{lem:basic_ILMQR}~\ref{lem:basic_ILMQR_p1}, it follows that
    $$
    D_{k}-\psi(x_{k+1}) \geq \varrho_{1}\left(q_{k}(0)-q_{k}(d_{k})\right) \geq \tfrac{\varrho_{1}\overline{\mu}_{k}}{2}\|d_{k}\|^{2} \geq \tfrac{\varrho_{1}\mu_{\min}}{2}\|d_{k}\|^{2} \geq 0,
    $$
    implying that $\|d_{k}\|\to 0$.
    In addition, by \Cref{pro:well_defineILMQR}~\ref{pro:well_defineILMQR_p2}, and residual condition, it can be obtained that
    $$0\leq \|\overline{r}_{k}\|\leq \tau\overline{\mu}_{k}\|d_{k}\|\leq \tau\tilde{\beta}\|d_{k}\|,$$ 
    verifying $\|\overline{r}_{k}\|\to 0.$
    % \begin{equation*}\label{eq:lim_r}
    %     \|\overline{r}_{k}\|\to 0.
    % \end{equation*}
    On the other hand, utilizing \Cref{lem:basic_ILMQR}~\ref{lem:basic_ILMQR_p2} and \Cref{pro:well_defineILMQR}~\ref{pro:well_defineILMQR_p2}, we deduce
    \begin{align*}
        D_{k}-\psi(x_{k+1}) &\geq \varrho_{1}\left(q(0)-q(d_{k})\right)
        \geq \tfrac{\varrho_{1}}{2\left(L_{0}^{2}+\overline{\mu}_{k}\right)} \|\overline{r}_{k}-\nabla \psi(x_{k})\|^{2}
        \geq \tfrac{\varrho_{1}}{2\left(L_{0}^{2}+\tilde{\beta}\right)} \Big(\|\overline{r}_{k}\|-\|\nabla \psi\left(x_{k}\right)\|\Big)^{2}\geq 0.
    \end{align*}
    Taking limits as $k\to\infty$ from the above inequality leads to
    $$
    \lim _{k \rightarrow \infty}\left\|\nabla \psi\left(x_{k}\right)\right\|=\lim _{k \rightarrow \infty}\left\|\nabla F\left(x_{k}\right) F\left(x_{k}\right)\right\|=0,
    $$
    adjusting the result.
\end{proof}
%%%%%%%%%%%

\Cref{thm:complexity} establishes the global and evaluation complexity bounds for the sequence $\{x_{k}\}_{k\ge 0}$ generated by ILMQR.
Let $N_i(\epsilon)$ and $N_{f}(\epsilon)$ denote, respectively, the total number of iterations and function evaluations required to obtain a point satisfying $\|\nabla \psi\left(x_{k}\right)\| \leq \epsilon$.

%%%%%%%%%%%%%%%%%%%%%%%%%%%%%%%%%%%%%%%%%%%%%%%%%%%%%%%%%%%
\begin{theorem}[Complexity analysis]\label{thm:complexity}
    Let \Cref{as:global} hold.
    If the sequence $\{x_{k}\}_{k\ge 0}$ is generated by ILMQR, then, for a given accuracy parameter $\epsilon>0$, the total number of iterations to guarantee $\|\nabla \psi(x_{k})\| \leq \epsilon$ is $N_i(\epsilon)\leq k_{0}$ in which
    \begin{equation}\label{eq:Ubound_N}
        k_{0}:=\left\lceil\tfrac{8(L_{0}^{2}+\tilde{\beta})}{\varrho_{1}(1-\theta_{\max })\epsilon^{2}}\psi\left(x_{0}\right) +1\right\rceil.
    \end{equation}
    Additionally,
    \begin{equation*}\label{eq:Ubound_N_f}
        N_{f}(\epsilon) \leq k_{0}\left(\tfrac{\log (\beta_{1})-\log \big((\xi_{\min}+\omega_{\min})\epsilon^{\eta}\big)}{\log \left(\nu_{1}\right)}\right).
    \end{equation*}
\end{theorem}
\begin{proof}
    We first justify $N_i(\epsilon)\leq k_{0}$. To reach our goal, it suffices to show that the total number of iterations required to ensure
    $\|\overline{r}_{k}\|\leq \tfrac{\epsilon}{2}$ and $\|\overline{r}_{k}-\nabla \psi(x_{k})\|\leq \tfrac{\epsilon}{2}$ is upper bounded by $k_{0}$.
    Suppose, for the sake of contradiction, that $\|\overline{r}_{k}\|> \tfrac{\epsilon}{2}$ or $\|\overline{r}_{k}-\nabla \psi(x_{k})\|> \tfrac{\epsilon}{2}$, for all $0\leq k\leq k_{0}$.
    Then, for each iteration $0\leq k\leq k_{0}$, from $\overline{\rho}_{k} \geq \varrho_{1}$, \Cref{lem:basic_ILMQR}~\ref{lem:basic_ILMQR_p1} and \ref{lem:basic_ILMQR_p2}, and the residual condition $\|\overline{r}_{k}\|\leq \tau\overline{\mu}_{k}\|d_{k}\|$, one obtains:
    $$
    \begin{aligned}
    D_{k}-D_{k+1} 
    & =
    \left(1-\theta_{k}\right)\left(D_{k}-\psi(x_{k+1})\right)
    \geq 
    \varrho_{1}\left(1-\theta_{k}\right)\left(q_{k}(0)-q_{k}(d_{k})\right) \\[2mm]
    & \geq \tfrac{\varrho_{1}\left(1-\theta_{\max }\right)}{2\left(L_{0}^{2}+\overline{\mu}_{k}\right)}\max\Big\{\|\overline{r}_{k}\|^{2}, \left\|\overline{r}_{k}-\nabla \psi\left(x_{k}\right)\right\|^{2}\Big\}
    > \tfrac{\varrho_{1}\left(1-\theta_{\max }\right)\epsilon^{2}}{8(L_{0}^{2}+\tilde{\beta})},
    \end{aligned}
    $$
    leading to a contradiction:
    $$
    \psi(x_{0})=D_{0} \geq D_{0}-D_{k_{0}}=\sum_{k=0}^{k_{0}-1}\left(D_{k}-D_{k+1}\right) > \tfrac{\varrho_{1}\left(1-\theta_{\max }\right)\epsilon^{2}}{8(L_{0}^{2}+\tilde{\beta})}k_{0}>\psi(x_{0}),
    $$
    by virtue of \eqref{eq:Ubound_N}.
    Therefore, $N_i(\epsilon)\leq k_{0}$. Furthermore, combining this bound with \eqref{eq:Ubound_s}, it follows that
    $$
    \begin{aligned}
    N_{f}(\epsilon) & \leq \sum_{k=0}^{N_{i}(\epsilon)-1} s_{k}
    \leq \sum_{k=0}^{N_{i}(\epsilon)-1} \frac{\log (\beta_{1})-\log \big((\xi_{\min}+\omega_{\min})\epsilon^{\eta}\big)}{\log \left(\nu_{1}\right)}\\[2mm]
    & \leq k_{0}\left(\frac{\log (\beta_{1})-\log \big((\xi_{\min}+\omega_{\min})\epsilon^{\eta}\big)}{\log \left(\nu_{1}\right)}\right),
    \end{aligned}
    $$
    giving our desired result.
\end{proof} 
%%%%%%%%%%%

%%%%%%%%%%%%%%
\begin{remark}
    There is an important observation regarding \Cref{as:global}~\ref{as:global1}: All results in \Cref{lem:basic_ILMQR}, \Cref{pro:well_defineILMQR}, and \Cref{thm:global_conv_ILMQR} remain valid if we replace the global Lipschitz continuity of $\nabla F$ in $\mathbb{R}^n$ with the assumption that $\nabla F$ is locally Lipschitz on $\mathbb{R}^n$.

    In fact, since $F$ is continuously differentiable and the level set $\Lambda(x_0)$ is bounded, $F$ is Lipschitz continuous on $\mathbb{B}(x_0,2\Lambda)\supseteq\Lambda(x_0)$. Consequently, there exists a constant $L_0 > 0$ such that
    $$\|F(x) - F(y)\|\leq L_{0}\|x-y\|,\quad \forall x,y\in \Lambda(x_{0}),$$
    $$\|\nabla F(x)\|\leq L_{0},\quad \forall x\in \Lambda(x_{0}).$$
    Additionally, for each approximate solution $d_{k}$ of system \eqref{eq:LMQR_direction} satisfying the residual condition $\|\overline{r}_{k}\|\leq \tau\overline{\mu}_{k}\|d_{k}\|$, we have
    $$\|d_{k}\| \leq \frac{L_{0}\|F(x_{0})\|}{\mu_{\min}(1-\tau)}=:\kappa,$$
    as established in \eqref{eq:Ubound_d}.
    The set $\Lambda(x_{0})+\{d\in \mathbb{R}^{m}: \|d\|\leq \kappa\}$ is also compact, and hence $\nabla F$ is Lipschitz continuous on this set with some modulus $L_{1}>0$.
    Therefore, for any such $d_{k}$, we come to
    $$\|\nabla F(x_{k})-\nabla F(x_k+d_{k})\|\leq L_{1} \|d_{k}\|.$$
    However, under this locally Lipschitz gradient assumption, a difficulty may arise in the complexity analysis of ILMQR given in \Cref{thm:complexity}, since the analysis relies on prior knowledge of a global Lipschitz constant.
\end{remark}
%%%%%%%%%%%%

%%%%%%%%%%%%%%%%%%%%%%%%%%%%%%%%%%%%%%%%%%%%%%%%%%%%%%%%%%%%%%%%%%%%%%%%%%%%%%%%%%%%%%%%%
%%%%%%%%%%%%%%%%%%%%%%%%%%%%%%%%%%%%%%%%%%%%%%%%%%%%%%%%%%%%%%%%%%%%%%%%%%%%%%%%%%%%%%%%%
\section{Preliminary numerical experiments}\label{sec.numapp}
% todo: add all comparison algo's
% todo: expand introduction

In this section, we describe the details of solving the linear systems that arise in the proposed algorithms and present the corresponding numerical experiments.
We begin with a brief overview of Krylov subspace methods, followed by a discussion of the advantages, disadvantages, and key differences between several iterative solvers, including CG, GMRES($\ell$), CGLS, and LSQR. We omit methods such as BiCG and BiCGStab($\ell$) as they lack theoretical convergence guarantees. Based on this discussion, we conclude that LSQR is the most suitable method for our setting.
Next, we compare the classical choices of the regularization parameter $\mu_k$ in inexact frameworks with our proposed adaptive variant. In particular, we test the choices $\mu_k=\|F(x_k)\|^2$, $\mu_k = \|F(x_k)\|$, $\mu_k = \|\nabla F(x_k)F(x_k)\|$, and the strategy from \cite{ahookhosh2020finding}.
The experiments involve nonlinear systems arising in biochemical reaction networks, as well as monotonic equations and other nonlinear mappings. Here, our method outperforms the others in the biochemical reaction network tests and performs competitively on the other problems.
We then evaluate an image deblurring task, where the objective is to reconstruct a blurred and noisy image. In this setting, we compare our methods with SpaRSA \cite{wright2009sparse} and TwIST \cite{bioucas2007new}. The results show only marginal differences in performance between the methods.
All experiments are run on an AMD Ryzen 7 6800H with 32GB of RAM.

To evaluate algorithmic performance, we take advantage of the Dolan and Mor\'e performance profile~\cite{dolan2002benchmarking} using performance metrics such as $N_i$, $N_f$, or $N_f+3N_i$. 
This profiling method compares the performance of each algorithm relative to the best observed performance on a per-problem basis. Let $\mathcal{S}$ denote the set of algorithms and $\mathcal{P}$ the set of test problems. For each problem $p\in \mathcal{P}$ and algorithm $s\in\mathcal{S}$, the performance ratio is~defined~as
\begin{equation*}
r_{p,s}:=\frac{t_{p,s}}{\min\{t_{p,s}:s\in\mathcal{S}\}},%\label{eq:perf}
\end{equation*}
in which $t_{p,s}$ is the computational result (that is, $N_i$, $N_f$ or $N_f+3N_i$).
If an algorithm fails to solve a problem $r_{p,s}$ is set to a failure threshold $r_\text{failed}$, chosen twice the maximum observed performance ratio. The overall performance profile of the algorithm $s$ for a factor
$\tau\in\mathbb{R}$ is given by
\[
\rho_{s}(\tau):=\frac{1}{n_{p}}\textrm{size}\{p\in\mathcal{P}:r_{p,s}\leq\tau\},
\]
representing the probability that $r_{p,s}$ is within the factor $\tau$
of the best possible ratio. The function $\rho_s(\tau)$ can be interpreted as a cumulative distribution function for the performance ratio. In particular, $\rho_s(1)$ indicates the probability that algorithm $s$ achieves the best performance, while $\displaystyle\lim_{\tau \to r_\mathrm{failed}} \rho_s(\tau)$ gives the fraction of problems successfully solved by $s$. This performance profile is a statistical measure of the efficiency among algorithms.
% In \Cref{fig:rps}, for example, $\tau$ is represented on the $x$-axis, and $P(r_{p,s}\leq\tau:1\leq s\leq n_{s})$ on the $y$-axis.

%%%%%%%%%%%%%%%%%%%%%%%%%%%%%%%%%%%%%%%%%%%%%%%%%%%%%%%%%%%%%%%%%%%%%%%%%%%%%%%%%%%%%%%%%
\subsection{{\bf Krylov subspace schemes for solving the linear subproblem}}

In each iteration, the ILM method computes the search direction by approximately solving the linear system \eqref{eq:LM_direction}, a step that often dominates the overall computational cost, especially in large-scale problems. Hence, the choice of an efficient iterative solver is crucial to the method's performance.
The available iterative solvers can generally be categorized into stationary methods and Krylov-type methods. Stationary methods, such as Chebyshev, Richardson, (weighted) Jacobi, and Gauss-Seidel, are rarely employed as standalone solvers in practice. Instead, they are commonly used as preconditioners or smoothers. In contrast, Krylov-type methods, including CG, BiCG, BiCGStab($\ell$), GMRES($\ell$), CGLS, and LSQR, possess favorable convergence and scalability properties, making them more suitable for the ILM framework.

At the basis of most state-of-the-art iterative solvers for linear systems lies the concept of a Krylov subspace.

\begin{definition}[Krylov subspace]
Given a matrix $A\in\mathbb{R}^{n\times n}$ and vector $v\in\mathbb{R}^n$, the Krylov sequence generated by $A$ and $v$ is given by
$\{A^kv\}_{k\ge 0}=\{v,~ Av,~ A^2v,~ \ldots\}$
and the $k$-dimensional Krylov subspace $\mathcal{K}_k(A,v)$ is defined~as
\[
\mathcal{K}_k(A,v)\coloneqq
\mathrm{span}
\left\{
v, Av, \ldots, A^{k-1}v
\right\}.
\]
\end{definition}
Most classical Krylov methods are based on solving a minimization problem over a $k$-dimensional Krylov-subspace. These methods aim to iteratively approximate the solution to the linear system
\(Ax = b\).

%%%%%%%%%%%%%%%%%%%%%%%%%%%%%%%%%%%%%%%%%%%%%%%%%%%%%%%%%%%%%%%%%%%%%%%%%%%%%%%%%%%%%%%%%
\subsubsection{{\bf The CG scheme}}
Let $A$ be a symmetric positive definite matrix. The Conjugate Gradient (CG) method generates, at each iteration, an approximate solution $x_k$ by solving the following minimization problem
\begin{equation*}
    \|x^*-x_k\|_A = \min_{x\in x_0+\mathcal{K}_k(A,r_0)}\|x^*-x\|_A,
\end{equation*}
where $x^*\in \R^n$ with $Ax^*=b$, \( r_0 = b - Ax_0 \) denotes the initial residual and $\|v\|_A \coloneqq \sqrt{v^TAv}$ is the energy norm.
Thus, the method ensures a monotonic decrease in error, often leading to fast convergence in practice~\cite{liesen2013krylov}.

%%%%%%%%%%%%%%%%%%%%%%%%%%%%%%%%%%%%%%%%%%%%%%%%%%%%%%%%%%%%%%%%%%%%%%%%%%%%%%%%%%%%%%%%%
\subsubsection{{\bf The GMRES scheme}}                                                                            
Let $A$ be an invertible matrix. The Generalized Minimal Residual (GMRES) method computes an approximate solution $x_k$ in iteration $k$ by solving the residual norm over the Krylov subspace minimization problem,
\begin{equation*}
\|Ax_k - b\|_2=\min_{x\in x_0+\mathcal{K}_k(A,r_0)} \|Ax - b\|_2.
\end{equation*}
This approach ensures that the residual's Euclidean norm is reduced at each step, making GMRES particularly suitable for solving general nonsymmetric linear systems \cite{liesen2013krylov}.

%%%%%%%%%%%%%%%%%%%%%%%%%%%%%%%%%%%%%%%%%%%%%%%%%%%%%%%%%%%%%%%%%%%%%%%%%%%%%%%%%%%%%%%%%
\subsubsection{{\bf The LSQR and CGLS schemes}}
Unlike CG and GMRES, the LSQR and Conjugate Gradient for Least Squares (CGLS) methods apply to rectangular systems where $A\in\mathbb{R}^{m\times n}$ with $m\neq n$. Both algorithms are mathematically equivalent to applying CG to normal equations $A^TAx = A^Tb$, yet exhibit better numerical stability, especially when $A$ is ill-conditioned.
These solvers are also well-suited for solving regularized least squares problems of the form:
\begin{align*}
\min_x \|Ax-b\|_2^ 2+\lambda^2 \|x\|^2_2,
\end{align*}
where $\lambda>0$ is a regularization parameter. The optimality conditions of this problem lead to the linear system $(A^TA+\lambda^2 I)x=A^Tb$, which coincides with the structure of the LM subproblem in \eqref{eq:LM_direction}, allowing LSQR and CGLS to be used directly within the LM framework. Moreover, this regularized problem can be equivalently expressed as
\begin{equation*}
\min_x
\left\|
\begin{bmatrix}
A\\
\lambda I
\end{bmatrix}
x
-
\begin{bmatrix}
b\\
0
\end{bmatrix}
\right\|_2,
\end{equation*}
which can be solved efficiently using LSQR-type methods optimized for numerical robustness \cite{paige1982lsqr,paige1982algorithm}.

%%%%%%%%%%%%%%%%%%%%%%%%%%%%%%%%%%%%%%%%%%%%%%%%%%%%%%%%%%%%%%%%%%%%%%%%%%%%%%%%%%%%%%%%%
\subsubsection{{\bf Comments on Krylov subspace schemes for solving \eqref{eq:LM_direction}}}
% While our experiments do not make use of this fact yet, 
Krylov-subspace methods are {\it matrix-free}, relying only on products with $Ax$ and/or $A^Tx$. This feature not only simplifies implementation ($\nabla F$ does not need to be formed explicitly), but also facilitates efficient parallelization.
Among the considered solvers, LSQR offers the most favorable trade-offs for solving the LM subproblem \eqref{eq:LM_direction}. In contrast, CG and GMRES require square systems, possibly leading to the explicit computation of $A^\top A$, which can be dense and ill-conditioned even when $A$ is sparse. Their convergence typically depends on the square of the condition number of $A$ and may behave similarly to LSQR or CGLS but with reduced robustness. Furthermore, GMRES, in particular, is memory-intensive due to its full orthogonalization process, limiting its scalability. Although CGLS is mathematically equivalent to LSQR, empirical observations suggest that LSQR performs better, especially in the presence of damping. Therefore, all numerical experiments in this study employ LSQR. For further details, see \cite{paige1982lsqr,paige1982algorithm,liesen2013krylov}.

%%%%%%%%%%%%%%%%%%%%%%%%%%%%%%%%%%%%%%%%%%%%%%%%%%%%%%%%%%%%%%%%%%%%%%%%%%%%%%%%%%%%%%%%%%
\subsection{{\bf Application to biochemical reaction networks}}\label{sec:biochem}
%todo: Ni + 3Nf
We now evaluate the proposed methods on a set of 21 biochemical models. These problems are particularly challenging due to their exponential structure, often leading to erratic convergence behavior. For consistency, the same parameter choices are used across all problems; however, problem-specific tuning could further enhance performance. In particular, we observed that the decay rate of $\mu_k$ is very aggressive in some cases, potentially impacting convergence. The stopping criterion is set to $\|F(x_k)\| < 10^{-6}$ or a maximum of 100{,}000 iterations. We report both the termination residual $\|F(x_{N_i})\|$ and the composite performance metric $N_f + 3N_i$, where $N_f$ is the number of evaluations of $F$ and $N_i$ is the number of iterations. If the stopping criterion is not met within the iteration budget, the run is labeled as ``failed''. As shown in \Cref{tab:my_label}, the ILLM method successfully converges on all problems except `iJP815', and exhibits small residuals across the board.

ILMQR and ILLM showed nearly identical performance in this setting. For instance, in `Ecoli\_core', both methods converged in 148 iterations; in `iAF692', ILMQR required 677 iterations versus 944 for ILLM. In both cases, the inner loop of ILMQR was triggered only once. Since the differences were negligible, we excluded ILMQR from further comparison and focused on ILLM versus the following established algorithms:
\begin{enumerate}
    \item[$\bullet$] ILMYF:\label{item:yf} Levenberg-Marquardt line search method with $\mu_k=\|F(x_k)\|^2$ from \cite{yamashita2001rate};
    \item[$\bullet$] ILMFY:Levenberg-Marquardt line search method with $\mu_k=\|F(x_k)\|$ from \cite{fan2005quadratic};
    \item[$\bullet$] ILevMar: Levenberg-Marquardt trust-region method with $\mu_k=\|\nabla F(x_k)F(x_k)\|$ from \cite{ipsen2011rank};
    \item[$\bullet$] ILLMAFV: Levenberg-Marquardt method from \cite{ahookhosh2020finding} (direct version of \Cref{ALGORITHM-2:illm}) with 
    \[\small \mu_k=\begin{cases}
        \max\left\{0.95\|F(x_k)\|^\eta + 0.05\|\nabla F(x_k)F(x_k)\|^\eta, 10^{-8}\right\} &~~~ \mathrm{if}~ 0.95^k>0.01,\\[2mm]
        \max\bigg\{\max\big\{0.95^k,10^{-10}\big\}\|F(x_k)\|^\eta + \left(1-\max\big\{0.95^k,10^{-10}\big\}\right)\|\nabla F(x_k)F(x_k)\|^\eta, 10^{-8}\bigg\} &~~~ \text{otherwise.}
    \end{cases}\]
\end{enumerate}
All methods share the same auxiliary parameters to ensure a fair comparison. Given that Jacobians are relatively dense in these problems, direct solvers were found to be more effective than LSQR in terms of both speed and accuracy, and are used in all methods except ILLM. The parameter choices for ILLM are
$\xi_k=0.5(0.9)^k,$ $\omega_k=0.5(0.9)^k,$ and $\eta=1.3$.
As seen in \Cref{fig:rps} and \Cref{tab:my_label}, ILLM outperforms the other methods in all performance metrics. Notably, ILMFY emerges as the second-best performing method in most instances.

\begin{figure}
    \centering
    
    \subfloat[Metric: $N_i$]{\includegraphics[width=.3\textwidth]{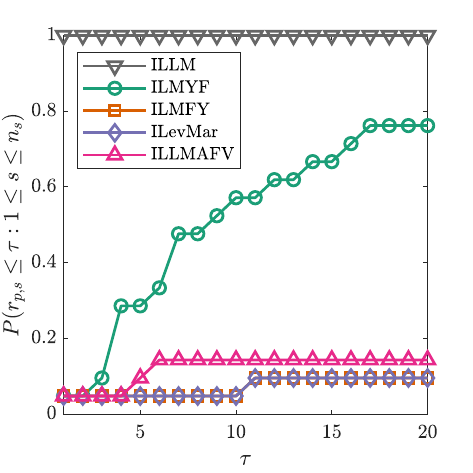}}
    \hspace{0.1cm}
    \subfloat[Metric: $N_f$]{\includegraphics[width=.3\textwidth]{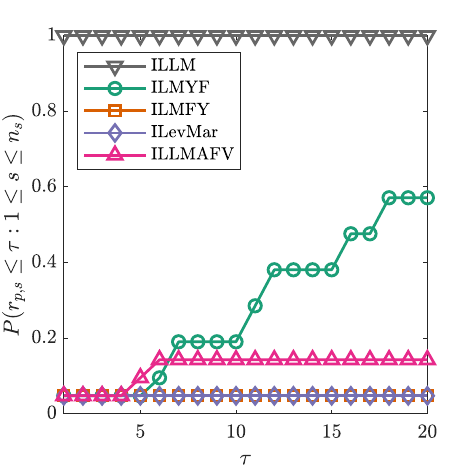}}
    \hspace{0.1cm}
    \subfloat[Metric: $N_f+3N_i$]{\includegraphics[width=.3\textwidth]{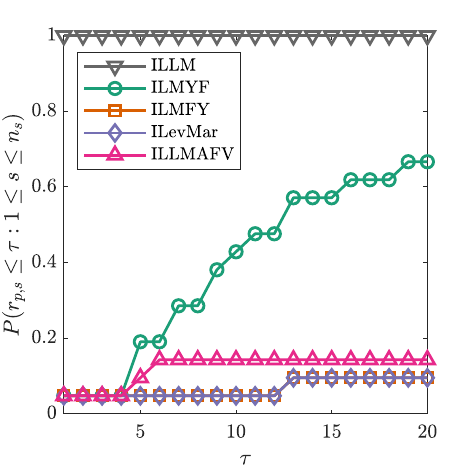}}
    \caption{Comparison between performance for ILLM, ILMYF, ILMFY, ILevMar, and ILLMAFV on 21 biochemical problems and three different metrics.}
    \label{fig:rps}
\end{figure}

%%%%%%%%%%%%%%%%%%%%%%%%%%%%%%%%%%%%%%%%%%%%%%%%%%%%%%%%%%%%%%%%%%%%%%%%%%%%%%%%%%%%%%%%%%
\subsection{{\bf Finding zeros of large-scale monotone mappings}}\label{sec:synth}
We now consider the problem of finding a zero of a monotone mapping $F:\mathbb{R}^n\to\mathbb{R}^n$ satisfying \((F(x)-F(y))^T(x-y)\geq 0,\) for all \(x,y\in\mathbb{R}^n.\)
Such mappings frequently arise as subproblems in generalized proximal algorithms and in the study of variational inequality problems \cite{bouaricha1998tensor,ortega2000iterative,zhao2001monotonicity}, making them an interesting class of problems in nonlinear analysis. The benchmark problems used in this section are listed in \Cref{tab:mappings_monotone}. Let us emphasize that all problems have extremely sparse Jacobians (tridiagonal or block-tridiagonal with tridiagonal blocks), which makes them easily scalable and well-suited for iterative linear solvers such as LSQR.

We compare the same set of algorithms as in \Cref{sec:biochem}. As shown in \Cref{fig:monotone}, ILLM achieves superlinear convergence in most cases. Although ILevMar occasionally performs slightly better, it fails to converge on several problems for which ILLM succeeds. This robustness is reflected in the performance profiles in \Cref{fig:perf_monotone}, where both ILLM and ILevMar consistently outperform the other methods. \Cref{tab:iters_monotone}, reporting $N_f + 3N_i$, shows that ILLM is the most robust, succeeding on all but one problem where the other methods frequently fail.

\begin{figure}[!htb]
    \centering
    \subfloat[Metric: $N_i$]{
    \includegraphics[width=.3\textwidth]{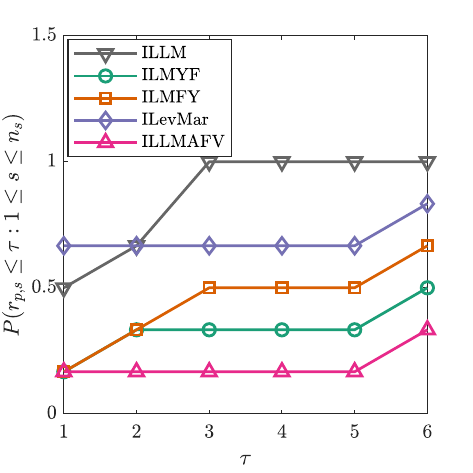}}
    \hspace{0.1cm}
    \subfloat[Metric: $N_f$]{
    \includegraphics[width=.3\textwidth]{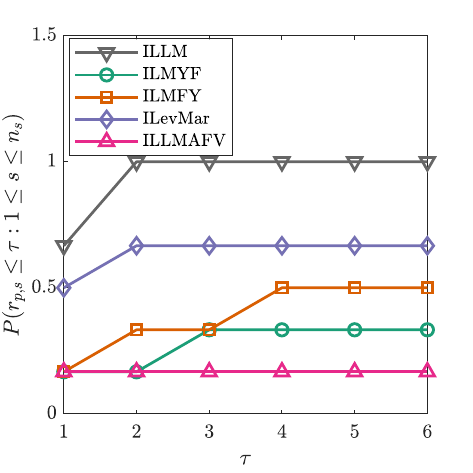}}
    \hspace{0.1cm}
    \subfloat[Metric: $N_f+3N_i$]{
    \includegraphics[width=.3\textwidth]{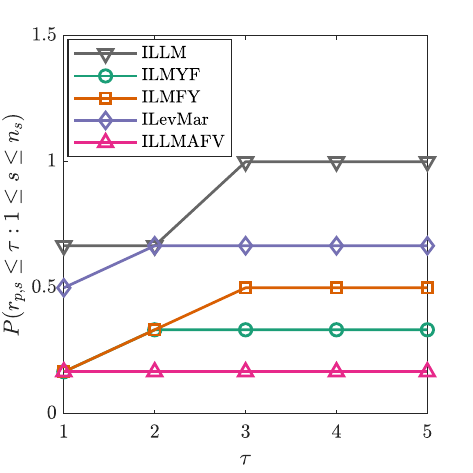}}
    \caption{Comparison between performance for ILLM, ILMYF, ILMFY, ILevMar, and ILLMAFV on six monotone mappings and three different metrics.}
    \label{fig:perf_monotone}
\end{figure}

\begin{table}[!htb]
    \centering
    \begin{tabular}{|c|c|c|c|c|c|}
    \hline
     Problem label  & ILLM           &ILMYF          &ILMFY              &ILevMar            & ILLMAFV   \\\hline
(a)     &35770          &failed     &failed    &failed     &failed\\ 
(b)     &  275          &failed     & 5326     &137        &21047\\ 
(c)     & 8790          &14236      &26302     &failed     &failed\\ 
(d)     &  110          &18219      &  168     &119        &failed\\ 
(e)     &failed         &failed     &failed    &failed     &failed\\ 
(f)     &310            &failed     &10390     &125        &37127\\ 
\hline
\end{tabular}
    \caption{$N_f+3N_i$ for every monotone mapping and algorithm.}
    \label{tab:iters_monotone}
\end{table}

\begin{figure}[!htb]
    \centering
    \centering
\subfloat[Convergence of function (a)]{\includegraphics[width=.3\textwidth]{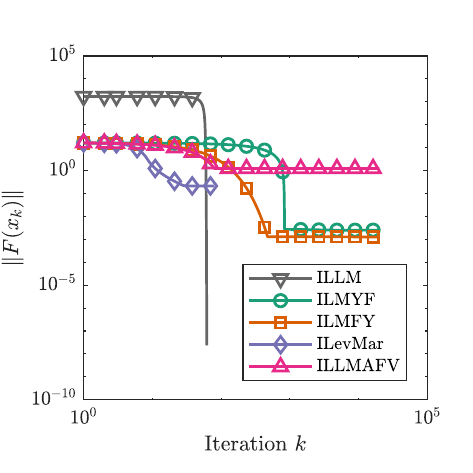}}%
\hspace{0.1cm}
\subfloat[Convergence of function (b)]{\includegraphics[width=.3\textwidth]{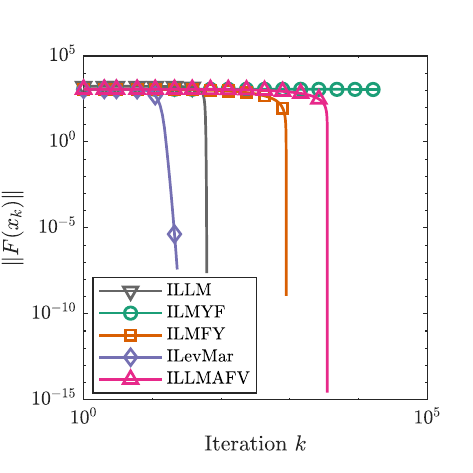}}%
\hspace{0.1cm}
\subfloat[Convergence of function (c)]{\includegraphics[width=.3\textwidth]{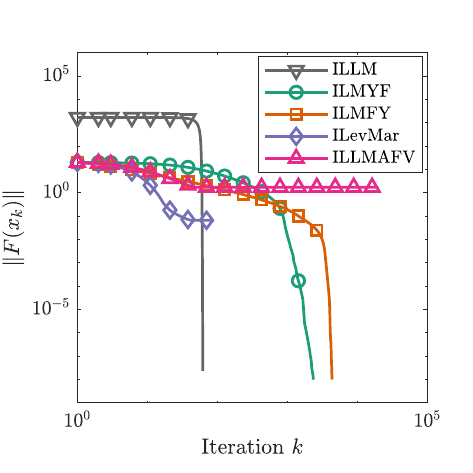}}%
\hspace{0.1cm}
\subfloat[Convergence of function (d)]{\includegraphics[width=.3\textwidth]{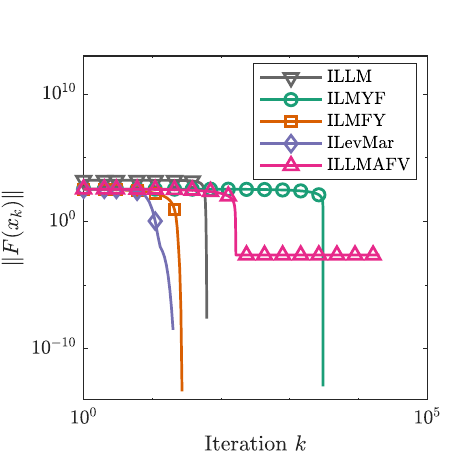}}%
\hspace{0.1cm}
\subfloat[Convergence of function (e)]{\includegraphics[width=.3\textwidth]{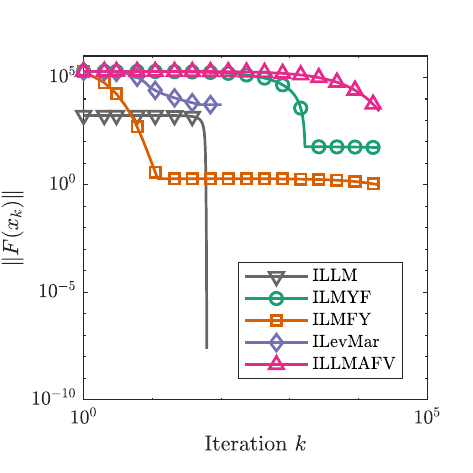}}%
\hspace{0.1cm}
\subfloat[Convergence of function (f)]{\includegraphics[width=.3\textwidth]{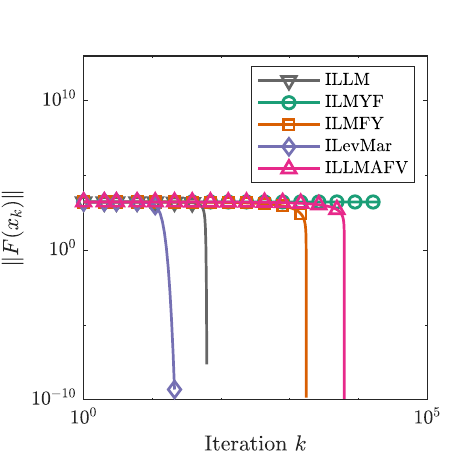}}%
\hspace{0.1cm}
    \caption{Convergence of $\|F(x_k)\|$ for six monotone mappings and five algorithms.}
    \label{fig:monotone}
\end{figure}

%%%%%%%%%%%%%%%%%%%%%%%%%%%%%%%%%%%%%%%%%%%%%%%%%%%%%%%%%%%%%%%%%%%%%%%%%%%%%%%%%%%%%%%%%%
\subsection{{\bf Finding zeros of nonmonotone mappings}}\label{sec:nonMap}
Our next set of experiments concerns a collection of mappings $F:\mathbb{R}^n\to \mathbb{R}^m$ characterized by an extremely sparse Jacobian.
The mappings are square (i.e., $n=m$) or overdetermined (i.e., $n>m$), and are not necessarily monotone. A detailed overview is provided in \Cref{tab:mappings}. We use the same parameters as in the monotone case.

As shown in \Cref{fig:synthetics}, our method converges superlinearly and successfully solves all problems. Consequently, \Cref{fig:perf_synth} illustrates that our approach consistently achieves the best performance. It is worth noting that ILevMar produced NaNs in two instances, for which we assigned the maximum iteration count of $10^3$. \Cref{tab:iters_mappings} reinforces this observation: ILLM is the only method that successfully solves all but one problem, while ILMFY performs competitively in a few cases-highlighting that our approach remains robust without compromising efficiency.

%\vspace{-3mm}
%%%%%%%%%%%%%%%%%
\begin{figure}[!htb]
\centering
\subfloat[Convergence of function (i)]{\includegraphics[width=.3\textwidth]{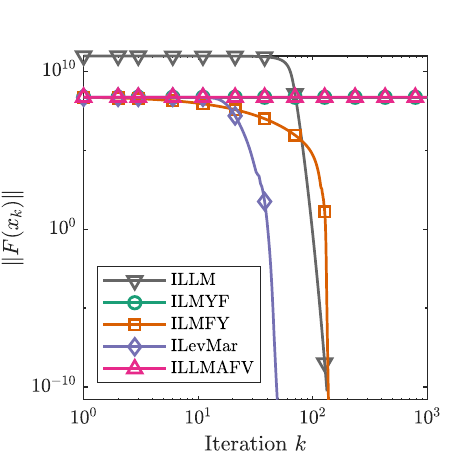}}%
\hspace{0.1cm}
\subfloat[Convergence of function (ii)]{\includegraphics[width=.3\textwidth]{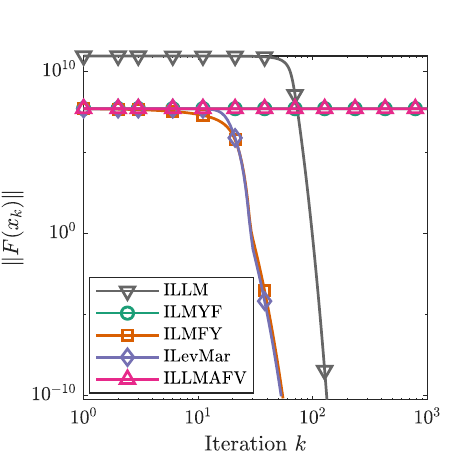}}%
\hspace{0.1cm}
\subfloat[Convergence of function (iii)]{\includegraphics[width=.3\textwidth]{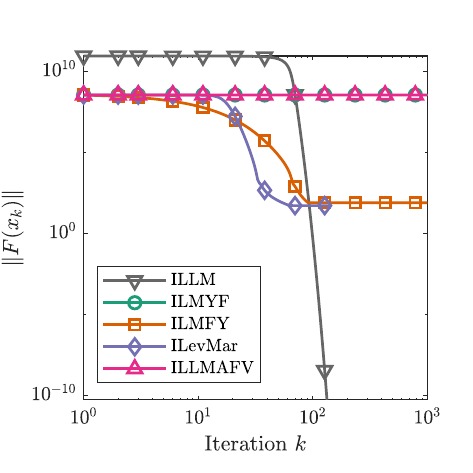}}%
\hspace{0.1cm}
\subfloat[Convergence of function (iv)]{\includegraphics[width=.3\textwidth]{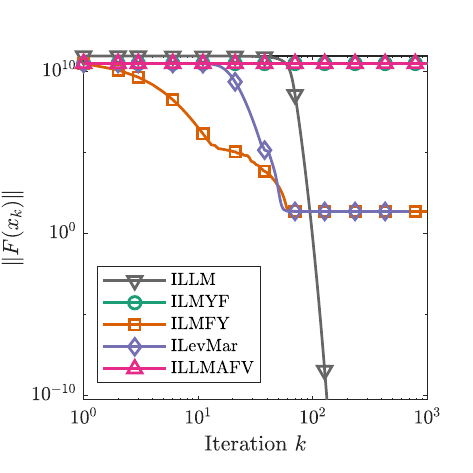}}%
\hspace{0.1cm}
\subfloat[Convergence of function (v)]{\includegraphics[width=.3\textwidth]{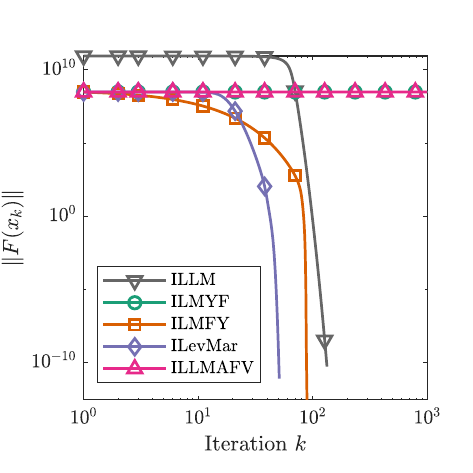}}%
\hspace{0.1cm}
\subfloat[Convergence of function (vi)]{\includegraphics[width=.3\textwidth]{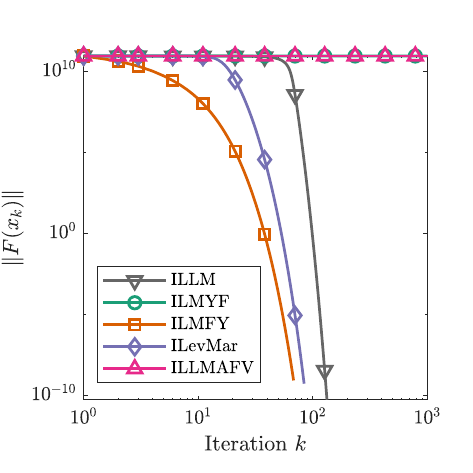}}%
\hspace{0.1cm}
\caption{Convergence of $\|F(x_k)\|$ for six nonlinear mappings and five algorithms.}
    \label{fig:synthetics}
\end{figure}

\begin{figure}[!htb]
    \centering
    \subfloat[Metric: $N_i$]{\includegraphics[width=.3\textwidth]{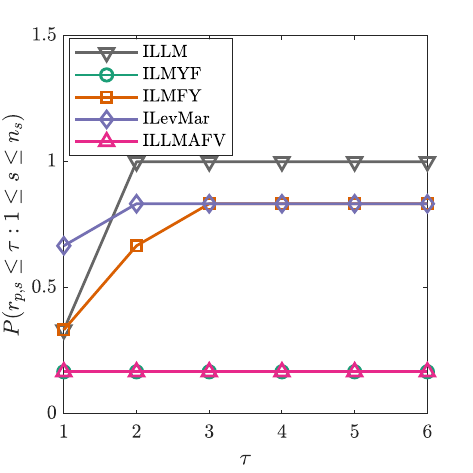}}
    \hspace{0.1cm}
    \subfloat[Metric: $N_f$]{\includegraphics[width=.3\textwidth]{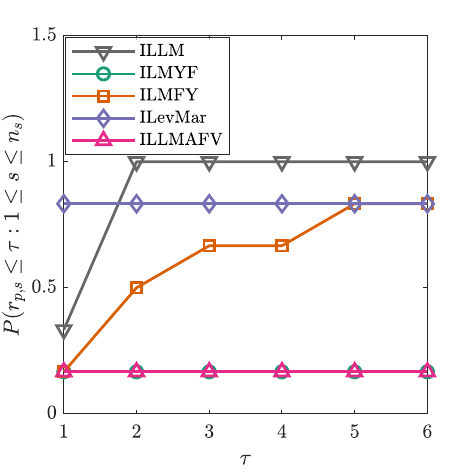}}
    \hspace{0.1cm}
    \subfloat[Metric: $N_f+3N_i$]{\includegraphics[width=.3\textwidth]{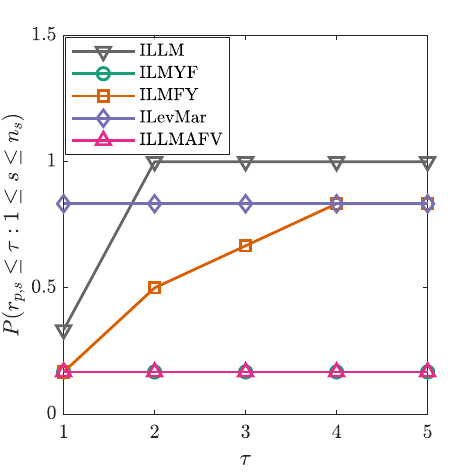}}
    \vspace{-2mm}
    \caption{Comparison among ILLM, ILMYF, ILMFY, ILevMar, and ILLMAFV on three different metrics.}
    \label{fig:perf_synth}
\end{figure}

\begin{table}[h]
    \centering
    \begin{tabular}{|c|c|c|c|c|c|}
    \hline
         Problem label&ILLM &ILMYF      &ILMFY      &ILevMar    & ILLMAFV   \\\hline
         (i)    & 470       &failed     & 832       & 245       &failed        \\
         (ii)   & 520       &failed     & 364       & 265       &failed        \\
         (iii)  & 635       &failed     &failed     &failed     &failed        \\
         (iv)   &failed     &failed     &failed     &failed     &failed        \\
         (v)    & 445       &failed     & 542       & 255       &failed        \\
         (vi)   & 665       &failed     & 446       & 420       &failed        \\
    \hline
    \end{tabular}
    \
    \caption{$N_f+3N_i$ for every nonlinear mapping and algorithm.}
    \label{tab:iters_mappings}
\end{table}

%%%%%%%%%%%%%%%%%%%%%%%%%%%%%%%%%%%%%%%%%%%%%%%%%%%%%%%%%%%%%%%%%%%%%%%%%%%%%%%%%%%%%%%
\subsection{{\bf Application to image deblurring}}

We conclude with image deblurring experiments on $256\times 256$ grayscale images corrupted by 2\% Gaussian noise, comparing our method against SpaRSA and TwIST methods \cite{wright2009sparse,bioucas2007new}. 
In the blurring process, each pixel in an image $X\in\mathbb{R}^{n\times m}$ is replaced by a weighted average of neighbourhood pixels, yielding a blurred image $B\in\mathbb{R}^{n\times m}$. Vectorizing both images to obtain $x, b\in\mathbb{R}^{nm}$, this process become a linear system $Ax=b$
where $A\in\mathbb{R}^{nm\times nm}$ is a sparse, ill-conditioned blurring matrix derived from a point spread function (PSF) \cite{hansen2006deblurring}.

Given the ill-posed nature of this problem, solving the inverse problem directly is unstable. Instead, we solve the regularized normal equations:
\begin{equation*}
    A^TAx - A^Tb + \gamma x = 0,
\end{equation*}
with $\gamma$ set to the noise level. In our experiments, we generate blurred data using IRTools \cite{gazzola2019ir}, applying a mild $15^\circ$ rotation blur matrix $A\in\mathbb{R}^{256^2\times 256^2}$ to a known image, vectorized into $x^*\in\mathbb{R}^{256^2}$, and adding Gaussian noise as
\begin{equation*}\label{eq:imageDeb}
    b = Ax^* + \mathcal{N}(0,0.02),
\end{equation*}
where $Ax^*\in\mathbb{R}^{256^2}$ is the blurred (vectorized) image and $\mathcal{N}(0,0.02)$ is a vector with entries drawn from the normal distribution with mean 0 and standard deviation 0.02 \cite{hansen2006deblurring,hansen2010discrete}.
To generate the images and blurring matrix in MATLAB, we utilize the IRTools package. We use default settings for SpaRSA and TwIST.

We evaluate performance using Root Mean Squared Error (RMSE) and Peak Signal-to-Noise Ratio (PSNR) given by
\begin{equation*}
    {\rm RMSE}:=\sqrt{\frac{1}{mn}\sum_{i=1}^{mn}\left(x^*_i - \tilde{x}_i\right)^2},\quad\quad
    {\rm PSNR}:=10\log_{10}\left(\frac{M^2}{\text{RMSE}^2} \right).
\end{equation*}
where $\tilde{x}$ is a recovered
image and $M$ denotes the maximum pixel intensity (in our case $M=1$).

\Cref{tab:snr} summarizes results on seven test cases. While performance differences are subtle, TwIST excels on images with large dark regions and low smoothness, whereas our method is better on smoother textures. Visual comparisons in \Cref{fig:deblur_imgs} show clearer structural recovery, especially in the `\texttt{circuit}' image, by our method, with better-defined traces than those recovered by SpaRSA or TwIST. Despite noise, both objective values and gradients are small at the termination.

%\begin{figure}[!htb] 
%\centering
%\subfloat{\includegraphics[width=.3\textwidth]{Figs/ilmar_synthetic/performanceNi3.pdf}}
%\hspace{0.1cm}
%\subfloat[Metric: $N_f$]{\includegraphics[width=.3\textwidth]{Figs/ilmar_synthetic/performanceNf3.pdf}}
%\hspace{0.1cm}
%\subfloat[Metric: $N_f+3N_i$]{\includegraphics[width=.3\textwidth]{Figs/ilmar_synthetic/performanceNiNf4.pdf}}
%\vspace{-2mm}
%\caption{Comparison among ILLM, ILMYF, ILMFY, ILevMar, and ILLMAFV on three different metrics.}
%\label{fig:deblur_imgs}
%\end{figure}

%	\newcommand*\mytablecontents{}
%\foreach \img in {hst,satellite,pattern2,smooth,ppower,cameraman,circuit}{
 % \foreach \type in {exact, blur, ilmar, sparsa, twist}{
  %  \xappto\mytablecontents{\includegraphics[width=.15\textwidth]{\img_\type_256.png} &}
 % }
%  \gappto\mytablecontents{\\}
%}

%\begin{figure}
%    \centering
%    \begin{tabular}{ccccc}
%        Exact & Blur & ILMQR & SpaRSA & TwIST \\
%        \mytablecontents
%    \end{tabular}
%    \caption{Comparison between ILMQR, SpaRSA, and TwIST for image deblurring with 2\% Gaussian noise. Upon visual inspection, SpaRSA performs worse, as confirmed by the RMSE.}
%    \label{fig:deblur_imgs}
%\end{figure}

\begin{figure}[htbp]
    \centering
    \begin{tabular}{ccccc}
        Exact & Blur & ILMQR & SpaRSA & TwIST \\
        \includegraphics[width=0.15\textwidth]{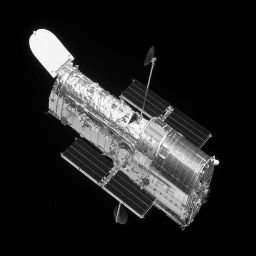} &
        \includegraphics[width=0.15\textwidth]{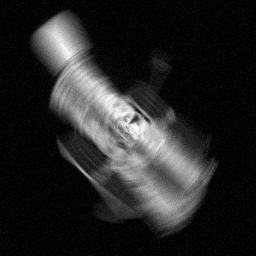} &
        \includegraphics[width=0.15\textwidth]{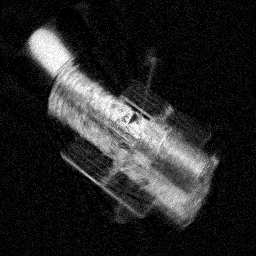} &
        \includegraphics[width=0.15\textwidth]{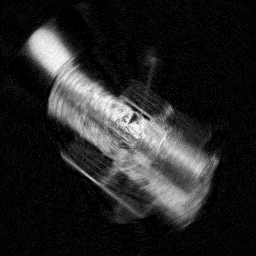} &
        \includegraphics[width=0.15\textwidth]{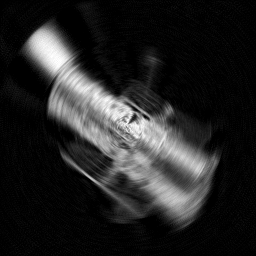} \\
        
        \includegraphics[width=0.15\textwidth]{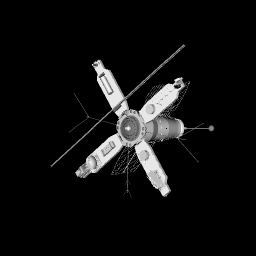} &
        \includegraphics[width=0.15\textwidth]{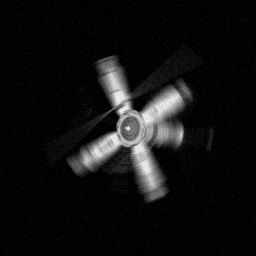} &
        \includegraphics[width=0.15\textwidth]{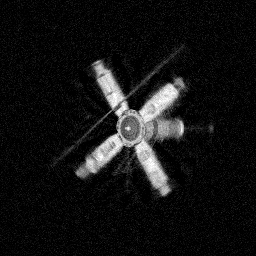} &
        \includegraphics[width=0.15\textwidth]{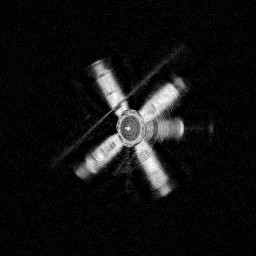} &
        \includegraphics[width=0.15\textwidth]{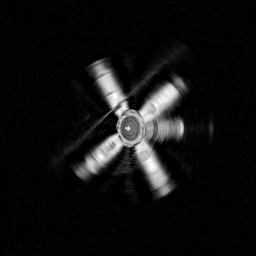} \\
        
        \includegraphics[width=0.15\textwidth]{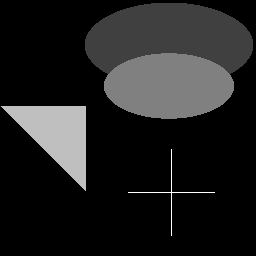} &
        \includegraphics[width=0.15\textwidth]{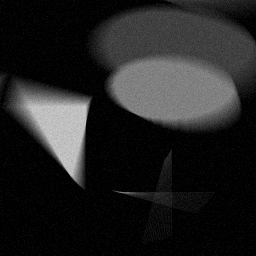} &
        \includegraphics[width=0.15\textwidth]{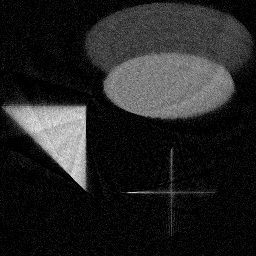} &
        \includegraphics[width=0.15\textwidth]{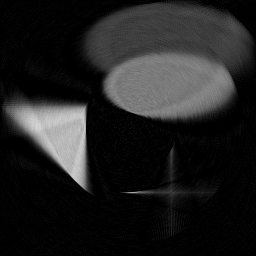} &
        \includegraphics[width=0.15\textwidth]{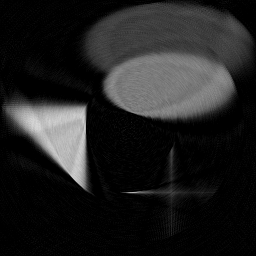} \\
        
        \includegraphics[width=0.15\textwidth]{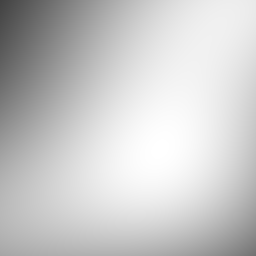} &
        \includegraphics[width=0.15\textwidth]{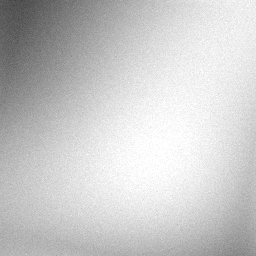} &
        \includegraphics[width=0.15\textwidth]{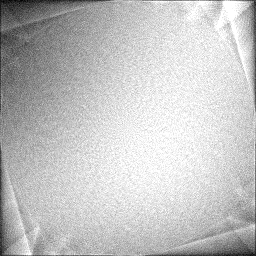} &
        \includegraphics[width=0.15\textwidth]{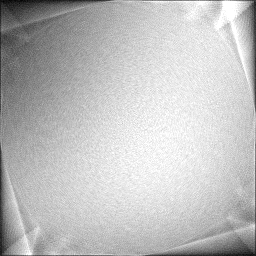} &
        \includegraphics[width=0.15\textwidth]{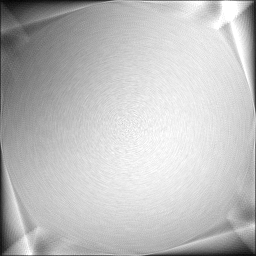} \\
        
        \includegraphics[width=0.15\textwidth]{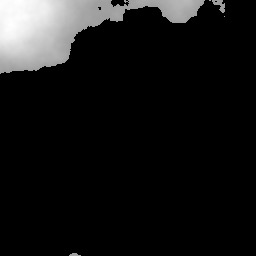} &
        \includegraphics[width=0.15\textwidth]{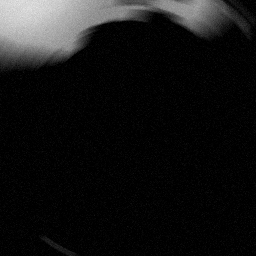} &
        \includegraphics[width=0.15\textwidth]{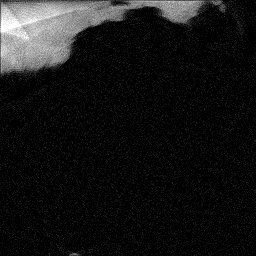} &
        \includegraphics[width=0.15\textwidth]{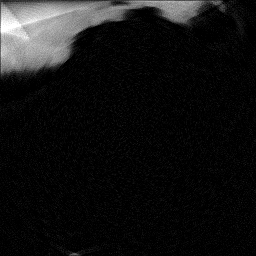} &
        \includegraphics[width=0.15\textwidth]{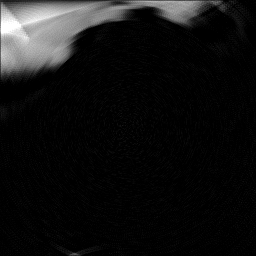} \\
        
        \includegraphics[width=0.15\textwidth]{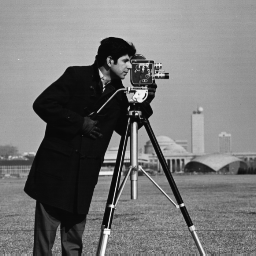} &
        \includegraphics[width=0.15\textwidth]{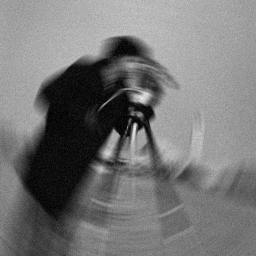} &
        \includegraphics[width=0.15\textwidth]{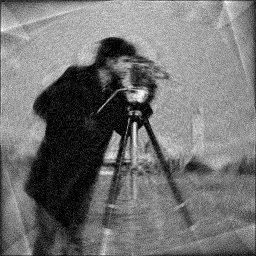} &
        \includegraphics[width=0.15\textwidth]{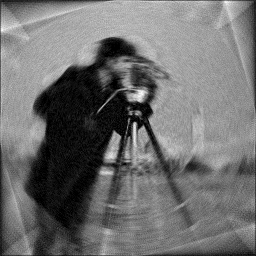} &
        \includegraphics[width=0.15\textwidth]{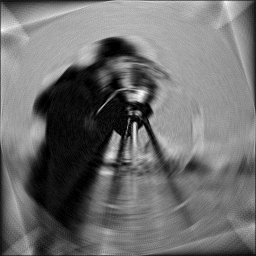} \\
        
        \includegraphics[width=0.15\textwidth]{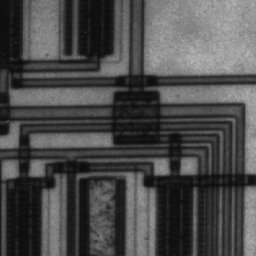} &
        \includegraphics[width=0.15\textwidth]{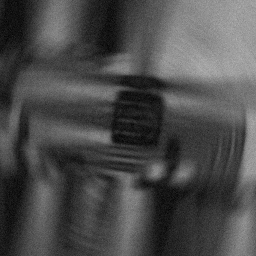} &
        \includegraphics[width=0.15\textwidth]{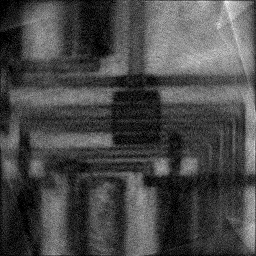} &
        \includegraphics[width=0.15\textwidth]{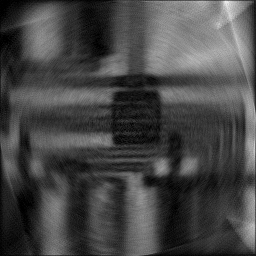} &
        \includegraphics[width=0.15\textwidth]{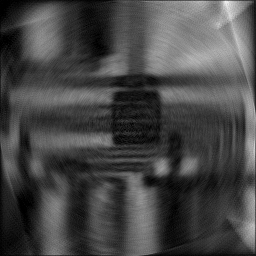} \\
    \end{tabular}
    \caption{Comparison between ILMQR, SpaRSA, and TwIST for image deblurring with 2\% Gaussian noise. Upon visual inspection, SpaRSA performs worse, as confirmed by the RMSE.}
    \label{fig:deblur_imgs}
\end{figure}

\begin{table}[h]
\begin{center}
    \begin{tabular}{|l|ccc|ccc|}
    \hline
        &       &   RMSE    &       
        &       & PSNR      &       
        \\ \hline
    Image label            
        & ILMQR  & SpaRSA    & TwIST 
        & ILMQR  & SpaRSA    & TwIST 
        \\ \hline
    \texttt{hst}    
        & 0.07 & \textbf{0.07}  & 0.07 
        & {22.8 dB} & \textbf{23.2 dB} & {23.0 dB} 
        % & 6.57 dB & 0.50 dB & \textbf{7.43 dB} 
        \\
    \texttt{satellite} 
        & 0.07 & \textbf{0.06}  & {0.06} 
        & 23.4 dB & \textbf{24.1 dB} & {24.0 dB} 
        % & 5.61 dB & 3.61 dB & \textbf{7.62 dB} 
        \\
    \texttt{pattern2} 
        & 0.07 & 0.07 & \textbf{0.07} 
        & 22.7 dB & 23.0 dB & \textbf{23.1 dB} 
        % & 4.67 dB & -2.17 dB & \textbf{5.71 dB}
        \\
    \texttt{smooth} 
        &\textbf{0.09} & 0.09 & 0.10 
        & \textbf{21.4 dB} & 20.9 db & 20.2 dB 
        % & \textbf{-9.07 dB} & -26.0 dB & -12.0 dB 
        \\
    \texttt{ppower} 
        & \textbf{0.09} & 0.09 & 0.09 
        & \textbf{21.3 dB} & 21.3 dB & 18.7 dB 
        % & \textbf{2.08 dB} & -6.70 dB & 0.70 dB 
        \\
    \texttt{cameraman} 
        & \textbf{0.10} & 0.11 & 0.11 
        & \textbf{19.9 dB} & 19.4 dB & 18.8 dB 
        % & \textbf{2.80 dB} & -10.4 dB & 0.50 dB
        \\
    \texttt{circuit} 
        & \textbf{0.08} & 0.08 & 0.08 
        & \textbf{22.2 dB} & 21.9 dB & 21.7 dB \\
        % & \textbf{4.48 dB} & -7.81 dB & 2.60 dB
    \hline
 \end{tabular}
\end{center}
 \caption{Comparison in RMSE and PSNR between ILMQR, SpaRSA, and TwIST for all seven deblurring problems.}
 \label{tab:snr}
\end{table}

%%%%%%%%%%%%%%%%%%%%%%%%%%%%%%%%%%%%%%%%%%%%%%%%%%%%%%%%%%%%%%%%%%%%%%%%%%%%%%%%%%%%%%%%%
%%%%%%%%%%%%%%%%%%%%%%%%%%%%%%%%%%%%%%%%%%%%%%%%%%%%%%%%%%%%%%%%%%%%%%%%%%%%%%%%%%%%%%%%%
%\vspace{-7mm}
\section{Conclusions}\label{sec:conclusion}
We proposed two inexact Levenberg–Marquardt algorithms -- ILLM and ILMQR -- for finding zeros of continuously differentiable mappings. The regularization parameter $\mu_k$ is defined adaptively, offering flexibility to tailor it for specific problem structures. For the ILLM method,  under H\"{o}lder metric subregularity and local H\"{o}lder gradient continuity, we established local (superlinear) convergence properties, with particular attention to the tolerance required for the iterative linear solver. To ensure global convergence, we introduced the ILMQR algorithm, which incorporates a nonmonotone quadratic regularization strategy and a modified update for $\mu_k$. We proved the well-definedness of the method, established global convergence, and provided a complexity analysis under Lipschitz gradient continuity. 
Since the proposed methods are inexact -- relying on iterative solvers to approximately solve the linear systems -- we provided a brief overview of Krylov-type methods, which remain the state of the art for solving sparse linear systems. Among the commonly used methods, LSQR emerged as the most suitable for our setting due to its favorable conditioning and reliance solely on matrix-vector products with $\nabla F$ and its transpose. We evaluated the proposed approach on a diverse set of problems to demonstrate its effectiveness and practical performance.

\section*{Appendix}
This appendix provides supplementary tables referenced in the main text. \Cref{tab:my_label} lists all models used in the biochemical reaction network experiments, together with the corresponding results for the algorithms under comparison. \Cref{tab:mappings_monotone,tab:mappings} detail the monotone and nonmonotone mappings considered in \Cref{sec:synth,sec:nonMap}, respectively.
\begin{landscape}
\renewcommand{\arraystretch}{1.4}
\begin{table}
    \centering
\begin{tabular}{|l|cc|cc|cc|cc|cc|}
\hline
                    & ILLM     &   &ILMYF & &ILMFY & &ILevMar & &ILLMAFV & \\\hline
     Problem label           & $N_f+3N_i$& $\|F(x_{N_i})\|$& $N_f+3N_i$& $\|F(x_{N_i})\|$& $N_f+3N_i$& $\|F(x_{N_i})\|$& $N_f+3N_i$& $\|F(x_{N_i})\|$& $N_f+3N_i$& $\|F(x_{N_i})\|$  \\ \hline
     `Ecoli\_core'  &740        & 9.44e-07       &  3281         & 7.08e-09  &failed &8.53e-05 &failed &2.47e-04 &4180      &9.99e-07 \\
     `iAF692'       &4720       & 9.44e-07       &392112         &9.99e-07   &failed &3.41e-04 &failed &4.19e-04 &failed    &1.33e-06 \\
     `iAF1260'      &2085       & 3.06e{-05}     & 68970         &9.99e-07   &failed &1.13e-03 &failed &7.96e-04 &failed    &3.19e-06\\
     `iBsu1103'     &2105       & 2.16e-05       & 38362         &9.99e-07   &failed &4.45e-04 &failed &4.80e-04 &failed    &1.57e-06\\
     `iCB925'       &3995       & 6.49e{-05}     & 39174         &9.99e-07   &failed &7.01e-04 &failed &7.11e-04 &failed    &1.05e-06\\
     `iIT341'       &4540       & 9.98e{-06}     & 19923         &9.99e-07   &failed &4.09e-04 &failed &4.71e-04 &failed    &1.02e-06\\
     `iJN678'       &3490       & 2.16e-05       & 28206         &9.99e-07   &failed &4.42e-04 &failed &4.75e-04 &failed    &1.58e-06\\
     `iJN746'       &1505       & 9.93e{-06}     & 32064         &9.99e-07   &failed &5.56e-04 &failed &5.29e-04 &failed    &1.83e-06\\
     `iJO1366'      &5935       &9.99e-07        & 48949         &9.99e-07   &failed &1.37e-03 &failed &8.49e-04 &failed    &3.15e-06\\ 
     `iJP815'       &failed     &0.0255          &failed         &NaN        &failed &9.90e-04 &failed &NaN      &failed    &NaN\\ 
     `iJR904'       & 1305      &9.24e-07        & 37027         &9.99e-07   &failed &1.08e-03 &failed &7.64e-04 &failed    &2.54e-06\\ 
     `iMB745'       &95675      &9.99e-07        &602843         &9.95e-07   &failed &1.59e-04 &failed &2.87e-04 &392890    &3.17e-06\\ 
     `iNJ661'       & 3465      &9.99e-07        &127897         &9.99e-07   &failed &4.78e-04 &failed &4.90e-04 &failed    &2.07e-06\\ 
     `iRsp1095'     & 7590      &4.56e-05        & 76553         &9.99e-07   &failed &5.73e-04 &failed &5.11e-04 &failed    &1.65e-06\\ 
     `iSB619'       & 1240      &8.91e-07        & 26536         &9.99e-07   &failed &7.33e-04 &failed &6.87e-04 &failed    &1.35e-06\\ 
     `iTH366'       & 7180      &9.99e-07        & 46513         &9.99e-07   &failed &1.56e-04 &failed &2.70e-04 &306685    &9.99e-07\\ 
     `iTZ479\_v2'   & 1760      &9.72e-07        & 21195         &9.99e-07   &failed &2.43e-04 &failed &3.53e-04 &failed    &1.05e-06\\ 
     `iYL1228'      & 9180      &5.44e-06        &failed         &2.04e-06   &failed &1.62e-03 &failed &9.23e-04 &failed    &3.15e-06\\ 
     `L\_lactis\_MG1363'& 3655  &9.98e-07        &14638          &9.98e-07   &failed &7.45e-04 &failed &6.68e-04 &failed    &1.70e-06\\ 
     `Sc\_thermophilis\_rBioNet'& 2695 &9.89e-07 &42165          &9.99e-07   &failed &6.12e-03 &failed &9.38e-04 &failed    &2.57e-06\\ 
     `T\_Maritima'  & 1430      &9.63e-07        &18028          &9.99e-07   &failed &2.72e-04 &failed &3.68e-04 &failed    &1.08e-06 \\
     \hline

\end{tabular}\\
\vspace{1em}
    \caption{Performance $(N_f+3N_i)$ and final objective value for every solver and every biochemical reaction network problem.}
    \label{tab:my_label}
\end{table}
\end{landscape}

\begin{table}[h]
    \centering
      \renewcommand{\arraystretch}{1.5}
    \begin{tabular}{|c|c|c|l|c|}
    \hline
    Problem label \hspace{2mm} & Dimensions \hspace{2mm}& $n$ \hspace{2mm}& $F(x)_i$ \hspace{2mm}& Reference \\\hline
    % \item $F:\mathbb{R}^n\to\mathbb{R}^n:F(x)_i = \sqrt{i}(x_i-i)$
    (a)   & $\mathbb{R}^n\to\mathbb{R}^n$ &$10000$ & 
    $\begin{cases}
        2x_1 + \sin(x_1)-1 & (i=1)\\
        -2x_{i-1} + 2x_i + \sin(x_i) -1\\
        2x_n + \sin(2x_n) -1 & (i=n)
    \end{cases}$
    & \cite{li2011class} \\
    (b) & $\mathbb{R}^n\to\mathbb{R}^n$ &$1000000$ & $2x_i-\sin(x_i)$ & \cite{li2011class}\\
    (c) & $\mathbb{R}^n\to\mathbb{R}^n$ &$10000$ &
    $(Ax + 3h^2X - 10h^2)_i$ & \\
    & & & $h = \frac{1}{\sqrt{n}+1};\quad A = B\otimes I_{\sqrt{n}} + I_{\sqrt{n}}\otimes B$; &\\ 
    & & & $B = \begin{pmatrix}
        2   & -1    &       &  \\
        -1  & 2     &\ddots &  \\
            &\ddots &\ddots &-1\\
            &       &-1     &2 
    \end{pmatrix}; 
    \quad  X = \left(x_i^3\right)_i$
    & \cite{li2011class} \\
    (d) & $\mathbb{R}^n\to\mathbb{R}^n$ &$10000$ &
    $\begin{cases}
        \frac{5}{2}x_1 + x_{2} -1 & (i=1)\\
        x_{i-1} + \frac{5}{2}x_i + x_{i+1} -1\\
        x_{n-1} + \frac{5}{2}x_n -1 & (i=n) 
    \end{cases}$
    & \cite{cheng2009prp} \\
    % (e) & $\mathbb{R}^n\to\mathbb{R}^n$ &2500 & $x_i - \frac{1}{n}x_i^2 + \frac{1}{n}\sum_{j=1}^nx_j + i$ & \cite{yan2010globally}\\
    (e) & $\mathbb{R}^n\to\mathbb{R}^n$ &$10000$ & 
    $\begin{cases}
        \frac{1}{3}x_1^3 + \frac{1}{2}x_2^2 & (i=1)\\
        \frac{-1}{2}x_i^2 + \frac{i}{3}x_i^3 + \frac{1}{2}x_{i+1}^2\\
        \frac{-1}{2}x_n^2 + \frac{n}{3}x_n^3& (i=n) 
    \end{cases}$
    & \cite{yan2010globally} \\
    (f) & $\mathbb{R}^n\to\mathbb{R}^n$ &$1000000$ &
    $\begin{cases}
        x_1 - \mathrm{exp}\left(\cos\left(\frac{x_1+x_2}{n+1}\right)\right) & (i=1)\\
        x_i - \mathrm{exp}\left(\cos\left(\frac{x_{i-1}+x_i+x_{i+1}}{n+1}\right)\right) \\
        x_n - \mathrm{exp}\left(\cos\left(\frac{x_{n-1}+x_n}{n+1}\right)\right) & (i=n)
    \end{cases}$
    & \cite{yan2010globally}\\
    \hline
    \end{tabular}
    %\vspace{2mm}
    \caption{The set of monotone equations with related dimensions used in the numerical experiments.}
    \label{tab:mappings_monotone}
\end{table}

\begin{table}[!htb]
    \centering
    \begin{tabular}{|c|c|c|l|c|}
    \hline
Problem label & Dimensions & $n$ & $F(x)_i$ & Reference \\\hline
% \item $F:\mathbb{R}^n\to\mathbb{R}^n:F(x)_i = \sqrt{i}(x_i-i)$
(i)     &$\mathbb{R}^n\to\mathbb{R}^n$      &3000 &$x_i^2-i$ & \cite{dan2002convergence} \\
(ii)    &$\mathbb{R}^{2n}\to\mathbb{R}^n$   &1500 & $x_ix_{i+n}-\sqrt{i}$ & \cite{yin2024modified}\\
(iii)   &$\mathbb{R}^{2n}\to\mathbb{R}^n$   &1500& $(x_i + x_{n+i})(x_i + x_{n+i} - \sqrt{i})$ & \cite{yin2024modified}\\
(iv)    &$\mathbb{R}^{n}\to\mathbb{R}^n$    &3000& $\begin{cases}
    \sqrt{i}\ \mathrm{exp}\left(\frac{x_i + x_{i+1}}{n}\right) & \text{if } i \text{ odd}\\
    \sqrt{i}\ (x_{i-1}+x_{i})(x_{i-1}+x_{i} - 1) \hspace*{5mm}& \text{if } i \text{ even}
\end{cases}$ & \cite{dan2002convergence}\\
(v)     &$\mathbb{R}^{2n}\to\mathbb{R}^n$   &1500& $(3-2x_{2i-1})x_{2i-1}- 2\sin(x_{2i})+1 $ & \cite{yin2024modified}\\
(vi)    &$\mathbb{R}^{3n}\to\mathbb{R}^n$   &1000& $x_{i}x_{n+i}x_{2n+i}-\sqrt[4]{i}$ & \cite{yin2024modified}\\
\hline
\end{tabular}
    \caption{Nonlinear square and rectangular mappings used in the numerical experiments.}
    \label{tab:mappings}
\end{table}

%%%%%%%%%%%%%%%%%%%%%%%%%%%%%%%%%%%%%%%%%%%%%%%%%%%%%%%%%%%%%%%%%%%%%%%%%%%%%%%%%%%%%%%%%
%%%%%%%%%%%%%%%%%%%%%%%%%%%%%%%%%%%%%%%%%%%%%%%%%%%%%%%%%%%%%%%%%%%%%%%%%%%%%%%%%%%%%%%%%
	
% ~~~~~~~~~~~~~~~~~~~~~~~~~~~~~~~~~~~~~~~~~~~~~~~~~~~~~~~~~~~~~~~~~~~~~~~~~~~~~~~~~~~~~~~~~~~~~~~~~

%\ifarxiv
%    \bibliographystyle{plain}
%\else
%    \phantomsection
%    \addcontentsline{toc}{section}{References}
%    \bibliographystyle{spmpsci}
%\fi
\bibliographystyle{plain}
\bibliography{Bibliography}

\end{document}